\documentclass[11pt,reqno]{amsart}
\usepackage{diagrams}
\usepackage{amssymb}
\usepackage[pagewise]{lineno}
\numberwithin{equation}{section}
\usepackage{cite}

\theoremstyle{definition}
\newtheorem{thm}{Theorem}[section]
\newtheorem{cor}[thm]{Corollary}

\newtheorem{prop}[thm]{Proposition}

\newtheorem{rem}[thm]{Remark}
\newtheorem{note}[thm]{Notation}

\newtheorem*{conj}{Conjecture}

\DeclareMathOperator{\Hc}{\mathcal{H}om}

\DeclareMathOperator{\Ima}{\mathrm{Im}}

\DeclareMathOperator{\p3}{\mathbb{P}^3}

\DeclareMathOperator{\pr}{\mathrm{pr}}

\DeclareMathOperator{\mo}{\mathcal{O}}

\newcommand{\mr}[1]{\mathrm{#1}}
\newcommand{\mb}[1]{\mathbb{#1}}
\newcommand{\mbf}[1]{\mathbf{#1}}
\newcommand{\mc}[1]{\mathcal{#1}}
\newcommand{\ov}[1]{\overline{#1}}

\newcommand{\mf}[1]{\mathfrak{#1}}

\begin{document}

\title[A conjecture of Mumford]{Generalization of a conjecture of Mumford}

\author[A. Dan]{Ananyo Dan}

\address{School of Mathematics and Statistics, University of Sheffield, Hicks building, Hounsfield Road, S3 7RH, UK}

\email{a.dan@sheffield.ac.uk}

\author[I. Kaur]{Inder Kaur}

\address{Pontifícia Universidade Cat\'{o}lica do Rio de Janeiro (PUC-Rio), R. Marqu\^{e}s de S\~{a}o Vicente, 225 - G\'{a}vea, Rio de Janeiro - RJ, 22451-900, Brazil}

\email{inder@mat.puc-rio.br}

\subjclass[2010]{Primary: $32$G$20$,  $32$S$35$, $14$D$07$,  $14$D$22$, $14$D$20$, Secondary: $14$H$60$, $55$R$40$}


\date{\today}

\begin{abstract}
 A conjecture of Mumford predicts a complete set of relations between the generators of the cohomology ring of the moduli space 
 of rank $2$ semi-stable sheaves with fixed odd degree determinant on a smooth, projective curve of genus at least $2$.
The conjecture was proven by Kirwan in $1992$. In this article, we generalize the conjecture  
to the case when the underlying curve is irreducible, nodal. In fact, we show that these relations (in the nodal curve case)
arise naturally as degeneration of the Mumford relations shown by Kirwan in the smooth curve case.
As a byproduct, we compute the Hodge-Poincar\'{e} polynomial 
of the moduli space of rank $2$, semi-stable, torsion-free sheaves with fixed determinant on an irreducible, nodal curve.
\end{abstract}

\maketitle

\section{Introduction}
The underlying field will always be $\mb{C}$.
Let $C$ be a smooth, projective curve of genus $g \geq 2$, $d$ be an odd integer and $\mc{L}_0$ be an invertible sheaf on $C$ of degree $d$.
Denote by $M_C(2,d)$ the moduli space of stable, locally-free sheaves of rank $2$ and degree $d$ over $C$
and by $M_C(2, \mc{L}_0)$ the sub-moduli space of $M_C(2,d)$ parameterizing 
locally-free sheaves  with determinant $\mc{L}_0$.
Since it was first constructed by Mumford, almost all aspects of the moduli spaces $M_C(2,d)$ and $M_C(2,\mc{L}_0)$ have been extensively studied.
Several different methods ranging from topology, number-theory, gauge theory as well as algebraic geometry 
have been used to study the cohomology ring $H^{*}(M_C(2,\mc{L}_0), \mb{Q})$. The generators of this 
ring were given by Newstead in \cite{new1} and a complete set of relations between these generators was conjectured by Mumford. 
We briefly recall the conjecture. 
Choose a symplectic basis $e_1, e_2, ..., e_{2g}$ of $H^1(C,\mb{Z})$ such that 
$e_i \cup e_j=0$ for $|j-i| \not= g$ and $e_i \cup e_{i+g}=-[C]^\vee$, where $[C]^\vee$ is the 
Poincar\'{e} dual of the fundamental class of $C$. Mumford and Newstead \cite{mumn} showed
that there exists an isomorphism of pure Hodge structures
\[\phi: H^1(C,\mb{Z}) \to H^3(M_C(2,\mc{L}_0), \mb{Z}),\]
induced by the second Chern class of the universal vector bundle $\mc{U}$ over $C \times M_C(2,\mc{L}_0)$.
Denote by $\gamma_i:=\phi(e_i)$ for $1 \le i \le 2g$ and $\gamma=\sum_{i=1}^g \gamma_i \gamma_{i+g}$. Newstead in \cite{new1}
showed that there exists $\alpha \in H^2(M_C(2,\mc{L}_0), \mb{Z})$ 
and $\beta \in H^4(M_C(2, \mc{L}_0), \mb{Z})$ (again arising from Chern classes of $\mc{U}$) such that the 
cohomology ring $H^*(M_C(2,\mc{L}_0), \mb{Q})$ is generated by 
$\alpha, \beta$ and $\gamma_i$ for $1 \le i \le 2g$.
Mumford conjectured that there is a decomposition
\[H^*(M_C(2, \mc{L}_0), \mb{Q}) \cong \bigoplus\limits_{k=0}^g P_k \otimes \mb{Q}[\alpha, \beta, \gamma]/I_{g-k}\]
where $I_k$ is an ideal of relations between $\alpha, \beta$ and $\gamma$ and $P_k$ is the primitive component of $\wedge^k
H^3(M_C(2,\mc{L}_0),\mb{Q})$ with respect to $\gamma$ (see \S \ref{subsecrel}
for precise definitions). 
The conjecture was proved by Kirwan \cite{kirw}.
In \cite{zag} Zagier showed that in fact the relations between the generators can be determined recursively. In particular,
$I_k \subset \mb{Q}[\alpha, \beta, \gamma]$ is generated by 
$(\xi_k, \xi_{k+1}, \xi_{k+2})$, where $\xi_0=1$ and recursively,
\[\xi_{k+1}:=\alpha \xi_k + k^2\beta \xi_{k-1}+2k(k-1)\gamma\xi_{k-2}.\]
This was also proven independently by Baranovskii \cite{bar}, Siebert and Tian \cite{sieb}, later by Herrera and Salamon \cite{HS} and also by King and Newstead \cite{kingn}. 
Although the obvious generalization of Mumford's conjecture to the cases when rank $n \geq 3$ is false, Earl and Kirwan in \cite{EaK} for arbitrary $n$, give additional relations such that 
together with the Mumford relations they form a complete set of relations between the generators of the cohomology ring of
the moduli space of rank $n$ semi-stable sheaves with coprime degree $d$ over $C$.
However, none of the existing literature studies the above conjecture for a singular curve, even in the case of rank $2$.

Let $X_0$ be an irreducible nodal curve with exactly one node.
Denote by $U_{X_0}(2,\mc{L}_0)$ the moduli space of rank $2$ semi-stable sheaves on $X_0$ with determinant $\mc{L}_0$ (here 
$\mc{L}_0$ is also an invertible sheaf of odd degree) as defined by Sun in \cite{sun1}
(we use a different notation for the moduli space as the definition of determinant in this case is different from the classical definition). We also know by \cite{K4} that the moduli space $U_{X_0}(2,\mc{L}_0)$ is non-empty. 
One of the difficulties in generalizing the above results to the cohomology ring of $U_{X_0}(2,\mc{L}_0)$ arises from the fact that 
$U_{X_0}(2,\mc{L}_0)$ is singular unlike the moduli space $M_C(2,\mc{L}_0)$.  As a result,
most of the techniques used for $M_C(2,\mc{L}_0)$ fail.

As there is no straightforward way to generalize the techniques in the literature, we 
instead embed the nodal curve $X_0$ in a regular family $\pi:\mc{X} \to \Delta$ (here $\Delta$ denotes the unit disc),
smooth over $\Delta^*:=\Delta\backslash \{0\}$ and central fiber isomorphic to $X_0$ (the existence of such a family follows from the 
completeness of the moduli space of stable curves, see \cite[Theorem B.$2$]{bake}). 
Note that the invertible sheaf $\mc{L}_0$ on $X_0$ lifts to a relative invertible sheaf $\mc{L}_{\mc{X}}$ over $\mc{X}$. 
There is a well-known relative Simpson's moduli space, denoted $U_{\mc{X}}(2,\mc{L}_{\mc{X}})$
of rank $2$ semi-stable sheaves with determinant $\mc{L}_{\mc{X}}$ over $\mc{X}$ (see \cite{ind, ink3} for basic definitions and results).
The (relative) moduli space $U_{\mc{X}}(2,\mc{L}_{\mc{X}})$ is flat over $\Delta$ and has central fiber $U_{X_0}(2,\mc{L}_0)$.
For any $s \in \Delta^*$, the fiber $U_{\mc{X}}(2,\mc{L})_s$ is isomorphic to $M_{\mc{X}_s}(2,\mc{L}|_{\mc{X}_s})$ and hence non-singular.
Substituting $C$ by $\mc{X}_s$ and $\mc{L}_0$ by $\mc{L}|_{\mc{X}_s}$ in the above discussion, we rewrite Mumford's conjecture 
for the cohomology ring $H^*(U_{X_0}(2,\mc{L}_0), \mb{Q})$ as follows:
\begin{conj}[Generalized Mumford conjecture]
 Denote by $P_k^{\mr{mon}}$ the subspace of $P_k$ (as before) consisting of all elements that are monodromy invariant (under the 
 natural monodromy action on $H^*(M_{\mc{X}_s}(2,\mc{L}|_{\mc{X}_s}), \mb{Q})$). Then, $P_k^{\mr{mon}}$ is independent (up to isomorphism)
 of the choice of the family $\pi$ and the cohomology ring $H^*(U_{X_0}(2,\mc{L}_0), \mb{Q})$
 decomposes as 
 \[H^*(U_{X_0}(2,\mc{L}_0), \mb{Q}) \cong \bigoplus\limits_{k=0}^g 
 P_k^{\mr{mon}} \otimes \mb{Q}[\alpha, \beta, \gamma]/I_{g-k}.\]
\end{conj}

Apart from the obvious motivation of Mumford's conjecture, one can also use the conjecture to 
 compute the Hodge-Poincar\'{e} polynomial
of the cohomology ring of $M_{\mc{X}_s}(2, \mc{L}_s)$ for $s \in \Delta^*$, where $\mc{L}_s:=\mc{L}_{\mc{X}}|_{\mc{X}_s}$ (see \cite{kingn}). 
Recall, the Hodge-Poincar\'{e} polynomial for $M_{\mc{X}_s}(2, \mc{L}_s)$: 
\[\sum\limits_{p,q} h^{p,q}(M_{\mc{X}_s}(2, \mc{L}_s),\mb{C})u^pv^q.\]
One of the first results in this direction was by Newstead \cite{newt}, where he gives a recursive formula for Betti numbers 
of $M_C(2,\mc{L}_0)$.
This was generalized by Harder, Narasimhan \cite{HN}, Desale and Ramanan \cite{DR} using number-theoretic methods, 
in the case of any coprime rank and degree.
Later, Bifet, Ghione and Letizia \cite{BGL} gave the same formula but using methods from algebraic geometry.
In \cite{EK}, Earl and Kirwan used methods from gauge theory to obtain the Hodge-Poincar\'{e} polynomial for $M_{\mc{X}_s}(2,d)$ and 
$M_{\mc{X}_s}(2, \mc{L}_s)$. However, an analogous Hodge-Poincar\'{e} polynomial for $U_{X_0}(2,\mc{L}_0)$ was yet unknown.
In this article we prove:
\begin{thm}\label{th:bet08}
The generalized Mumford conjecture holds true. Furthermore,
 the cohomology ring $H^*(U_{X_0}(2,\mc{L}_0), \mb{Q})$ is generated by $\alpha, \beta, \gamma_i$ for $1 \le i \le 2g-1$ and $\gamma_g\gamma_{2g}$.
\end{thm}
See Theorem \ref{th:simppo} and Remark \ref{rem:gen} for the precise statements.
 We also obtain the Hodge-Poincar\'{e} formula: As $U_{X_0}(2,\mc{L}_0)$ is singular, the associated 
 cohomology groups do not have a pure Hodge structure. Then, the 
 Hodge-Poincar\'{e} formula for $U_{X_0}(2,\mc{L}_0)$ is 
 defined as \[\sum\limits_{p,q} \sum\limits_i  \dim H^{p,q} \mr{Gr}^W_{p+q} H^i(U_{X_0}(2,\mc{L}_0),\mb{C})x^py^q.\]
 We also show:
\begin{thm}[see Theorem \ref{th:simppo}]\label{th:bet07}
 The Hodge-Poincar\'{e} polynomial associated to the moduli space 
 $U_{X_0}(2,\mc{L}_0)$ is 
 \[\frac{(1+xy^2)^{g-1}(1+x^2y)^{g-1}(1+xy+x^3y^3)-x^gy^g(1+x)^{g-1}(1+y)^{g-1}(2+xy)}{(1-xy)(1-x^2y^2)}.\]
\end{thm}

We now discuss the strategy of the proof.
The idea is to use the theory of variation of mixed Hodge structures by Schmid \cite{schvar} and Steenbrink \cite{ste1}
to relate the mixed Hodge structure on the central fiber of the 
relative moduli space to the limit mixed Hodge structure on the generic fiber.
Unfortunately, the singularity of the central fiber of $U_{\mc{X}}(2,\mc{L}_{\mc{X}})$
is not a normal crossings divisor, hence not a suitable 
candidate for using tools from \cite{ste1, schvar}. However, there is a different construction of a relative 
moduli space of rank $2$ semi-stable sheaves with determinant $\mc{L}_{\mc{X}}$ over $\mc{X}$, due to Gieseker \cite{gies}.
The advantage of the latter family of moduli spaces is that, in this case the central fiber, denoted $\mc{G}_{X_0}(2, \mc{L}_0)$ 
is a simple normal crossings divisor, hence compatible with the setup in 
\cite{ste1, schvar}. Moreover, the generic fibers
of the two relative moduli spaces coincide. 
We study the generators of the cohomology ring of $\mc{G}_{X_0}(2,\mc{L}_0)$
in a separate article \cite{genpre}, as it does not play any role in the proof of 
the generalized Mumford's conjecture.
Denote by $\mc{G}(2,\mc{L}_{\mc{X}})_\infty$ the generic fiber of the 
relative moduli space. By Steenbrink \cite{ste1}, $H^i(\mc{G}(2,\mc{L}_{\mc{X}})_\infty, \mb{Q})$
is equipped with a (limit) mixed Hodge structure such that
the specialization morphism \[\mr{sp}_i: H^i(\mc{G}_{X_0}(2,\mc{L}_0), \mb{Q}) \to H^i(\mc{G}(2,\mc{L}_{\mc{X}})_\infty, \mb{Q})\]
is a morphism of mixed Hodge structures for all $i \ge 0$.
Using the Mumford relations on $M_{\mc{X}_s}(2, \mc{L}_s)$ shown by Kirwan, 
we obtain a complete set of relations between the generators
of the cohomology ring $H^*(\mc{G}(2,\mc{L})_\infty, \mb{Q})$ of
the generic fiber $\mc{G}(2,\mc{L}_{\mc{X}})_\infty$.
However, not all elements in $H^i(\mc{G}(2, \mc{L}_{\mc{X}})_\infty, \mb{Q})$
are monodromy invariant. As a result the specialization morphism 
$\mr{sp}_i$ is neither injective, nor surjective (see Corollary \ref{ntor03}).
To compute explicitly the kernel and cokernel of $\mr{sp}_i$ we 
study the Gysin morphism from the intersection of the two 
components of $\mc{G}_{X_0}(2,\mc{L}_0)$ to the irreducible
components (see Theorem \ref{bet01}). Using this we prove:

\begin{thm}[see Theorem \ref{th:bet1}]\label{th:intro1}
 We have the following isomorphism of \emph{graded rings}:
 \[H^*(\mc{G}_{X_0}(2,\mc{L}_0),\mb{Q}) \cong \left(\bigoplus\limits_i P_i^{\mr{mon}} \otimes 
 \frac{\mb{Q}[\alpha, \beta, \gamma]}{I_{g-i}}\right) \oplus \left(\bigoplus\limits_i \widetilde{P}_{i-2} 
 \otimes \frac{\mb{Q}[\widetilde{\alpha}, \widetilde{\beta}, \widetilde{\gamma}, X, Y]}{(\widetilde{I}_{g-i-3}, X^2, Y^2, X-Y)}\right),\]
 where $\widetilde{P}_i, \widetilde{\alpha}, \widetilde{\beta}, \widetilde{\gamma}$ and $\widetilde{I}_{g-i}$ are objects analogous 
 to ${P}_i, {\alpha}, {\beta}, {\gamma}$ and ${I}_{g-i}$ defined earlier after replacing $C$ by $\widetilde{X}_0$ and $\mc{L}_0$ by
 $\widetilde{\mc{L}}_0:=\pi_0^*\mc{L}_0$, $\pi_0: \widetilde{X}_0 \to X_0$ is the normalization morphism.
 \end{thm}
 
 Similarly, we obtain the Hodge-Poincar\'{e} 
formula for $\mc{G}_{X_0}(2, \mc{L}_0)$ (see Theorem \ref{th:giespo}). Finally, 
there exists a proper morphism from $\mc{G}_{X_0}(2,\mc{L}_0)$ to $U_{X_0}(2,\mc{L}_0)$. Using this
  morphism we obtain from Theorem \ref{th:intro1}, the relations between the generators of the cohomology ring of
  $U_{X_0}(2,\mc{L}_0)$ and compute the Hodge-Poincar\'{e} polynomial for $U_{X_0}(2, \mc{L}_0)$.
  
 We remark that this article is part of a series of articles in which we study related but different questions pertaining to the moduli space of stable, rank 2 sheaves on an irreducible nodal curve (see \cite {indpre, DK2} for the first two published articles in the series). However, the results in all the articles are independent and overlap only in the background material. 

{\bf{Outline:}} In \S \ref{sec3}, we recall the preliminaries on the limit mixed Hodge structures applicable in our setup and use it to compute 
the limit mixed Hodge structure associated to the degenerating family $\pi$ of curves, mentioned above.
In \S \ref{sec2a}, we recall the relative Mumford-Newstead isomorphism as mentioned in \cite{indpre} which gives us an isomorphism (of mixed Hodge structures) between the limit mixed Hodge structure coming from $\pi$ and that coming from the associated family of moduli space of semi-stable sheaves.
In \S \ref{sec2b}, we compute a Gysin morphism to relate the cohomology ring of $\mc{G}_{X_0}(2,\mc{L}_0)$ to that of the generic 
fiber of the family of moduli spaces. In \S \ref{sec4}, we prove the generalized Mumford conjecture for $\mc{G}_{X_0}(2,\mc{L}_0)$.
In \S \ref{sec5}, we compute the Hodge-Poincar\'{e} polynomial for $\mc{G}_{X_0}(2,\mc{L}_0)$. In \S \ref{sec6}, 
we prove the generalized Mumford conjecture for $U_{X_0}(2,\mc{L}_0)$ and compute the associated Hodge-Poincar\'{e} polynomial.
  
{\bf{Notation:}} Given any morphism $f:\mc{Y} \to S$ and a point $s \in S$, we denote by $\mc{Y}_s:=f^{-1}(s)$.
 The open unit disc is denoted by $\Delta$ and $\Delta^*:=\Delta\backslash \{0\}$ denotes the punctured disc.
 
 \emph{Acknowledgements} 
We thank Prof. J. F. de Bobadilla and S. Basu for numerous discussions.
At the time of writing the article, the first author was supported by ERCEA Consolidator Grant $615655$-NMST and also
by the Basque Government through the BERC $2014-2017$ program and by Spanish
Ministry of Economy and Competitiveness MINECO: BCAM Severo Ochoa
excellence accreditation SEV-$2013-0323$ and the second author was funded by a fellowship from CNPq Brazil. Currently, the first 
author is funded by EPSRC grant number R/162871-11-1 and 
the second author is funded by a PNPD-CAPES fellowship, provided by PUC-Rio.

 \section{Preliminaries: Limit mixed Hodge structure}\label{sec3}
 
We recall basic results on limit mixed Hodge structures relevant to our setup. 
 See \cite{pet} for a detailed treatment of the subject.

 Let $\rho:\mc{Y} \to \Delta$ be a flat family of projective varieties, smooth over $\Delta^*$ and 
   $\rho':\mc{Y}_{\Delta^*} \to \Delta^*$ the restriction 
    of $\rho$ to $\Delta^*$. 
    
    \subsection{Hodge bundles}\label{nsec1}
     Using Ehresmann's theorem (see \cite[Theorem $9.3$]{v4}), we have for all $i \ge 0$, 
$\mb{H}_{\mc{Y}_{\Delta^*}}^i:=R^i{\rho}'_{*}\mb{Z}$
 the local systems over $\Delta^*$ with fiber $H^i(\mc{Y}_t,\mb{Z})$, for $t \in \Delta^*$.
 One can canonically associate to these local systems the holomorphic vector bundles 
 $\mc{H}_{\mc{Y}_{\Delta^*}}^i:=\mb{H}_{\mc{Y}_{\Delta^*}}^i \otimes_{\mb{Z}} \mo_{\Delta^*}$
 called the \emph{Hodge bundles}.
 There exist holomorphic sub-bundles $F^p\mc{H}_{\mc{Y}_{\Delta^*}}^i \subset \mc{H}_{\mc{Y}_{\Delta^*}}^i$
defined by the condition: for any $t \in \Delta^*$, the fibers \[\left(F^p\mc{H}_{\mc{Y}_{\Delta^*}}^i\right)_t \subset \left(\mc{H}_{\mc{Y}_{\Delta^*}}^i\right)_t\]
can be identified respectively with 
$F^pH^i(\mc{Y}_t,\mb{C}) \subset H^i(\mc{Y}_t,\mb{C})$, where $F^p$ denotes the Hodge filtration (see \cite[\S $10.2.1$]{v4}).
  
 \subsection{Canonical extension of Hodge bundles}
 The Hodge bundles and their holomorphic sub-bundles defined above can be extended to the entire disc. In particular,
 there exists a  \emph{canonical extension}, $\ov{\mc{H}}_{\mc{Y}}^i$ of 
 ${\mc{H}}_{\mc{Y}_{\Delta^*}}^i$ to $\Delta$ (see \cite[Definition $11.4$]{pet}).
 Note that $\ov{\mc{H}}_{\mc{Y}}^i$ is locally-free over $\Delta$. Denote by $j:\Delta^* \to \Delta$
 the inclusion morphism, 
 $F^p\ov{\mc{H}}_{\mc{Y}}^i:= j_*\left(F^p\mc{H}_{\mc{Y}_{\Delta^*}}^i\right) \cap  \ov{\mc{H}}_{\mc{Y}}^i$.
 Note that $F^p\ov{\mc{H}}_{\mc{Y}}^i$ is the \emph{unique largest} locally-free
 sub-sheaf of $\ov{\mc{H}}_{\mc{Y}}^i$ which extends $F^p\mc{H}_{\mc{Y}_{\Delta^*}}^i$.

 Consider the universal cover $\mf{h} \to \Delta^*$ of the punctured unit disc. 
 Denote by $e:\mf{h} \to \Delta^* \xrightarrow{j} \Delta$ the composed morphism and  
 $\mc{Y}_\infty$ the base change of the family $\mc{Y}$ over $\Delta$ to $\mf{h}$, by the morphism $e$.
   There is an explicit identification of the central fiber of the canonical extensions $\ov{\mc{H}}_{\mc{Y}}^i$ 
 and the cohomology group $H^i(\mc{Y}_{\infty},\mb{C})$, depending on the choice of the parameter $t$ on $\Delta$ (see \cite[XI-$8$]{pet}):
 \begin{equation}\label{tor23}
  g^i_{_t}:H^i(\mc{Y}_{\infty},\mb{C}) \xrightarrow{\sim} \left(\ov{\mc{H}}_{\mc{Y}}^i\right)_0.
 \end{equation}
 This induces (Hodge) filtrations on $H^i(\mc{Y}_{\infty},\mb{C})$ as
$F^pH^i(\mc{Y}_{\infty},\mb{C}):=(g_{_t}^i)^{-1}\left(F^p\ov{\mc{H}}_{\mc{Y}}^i\right)_0$.

\subsection{Monodromy transformations}
For any $s \in \Delta^*$ and $i \ge 0$, denote by 
\[T_{s,i}: H^i(\mc{Y}_s,\mb{Z}) \to H^i(\mc{Y}_s,\mb{Z}) \, \mbox{ and }\, T_{s,i}^{\mb{Q}}: H^i(\mc{Y}_s,\mb{Q}) \to H^i(\mc{Y}_s,\mb{Q})\]
the \emph{local monodromy transformations}
   defined by parallel transport along a counterclockwise loop about $0 \in \Delta$ (see \cite[\S $11.1.1$]{pet}).
 By \cite[Theorem II.$1.17$]{deli2}  (see also \cite[Proposition I.$7.8.1$]{kuli}) the automorphism extends to a $\mb{Q}$-automorphism 
\begin{equation}\label{int01}
 T_i: H^i(\mc{Y}_{\infty},\mb{Q}) \to H^i(\mc{Y}_{\infty},\mb{Q}).
\end{equation}
Denote by $T_{i,\mb{C}}$ the induced automorphism on $H^i(\mc{Y}_{\infty},\mb{C})$.

 \subsection{Schmid's limit mixed Hodge structures}
 The natural specialization morphism from the cohomology on the central fiber of the family $\mc{Y}$ to a general fiber $\mc{Y}_s$, $s \in \Delta^*$
 is not in general a morphism of mixed Hodge structures, if one considers the cohomology of $\mc{Y}_s$ with the natural pure Hodge structure.
 However, one can define a mixed Hodge structure on the cohomology of $\mc{Y}_s$, such that the specialization morphism is a morphism of mixed Hodge structures.
 More precisely,
  \begin{rem}\label{tor25}
  Let $N_i$ be the logarithm of the monodromy operator $T_i$.
  By \cite[Lemma-Definition $11.9$]{pet}, there exists an unique increasing \emph{monodromy weight filtration} $W_\bullet$ on $H^i(\mc{Y}_\infty,\mb{Q})$ such that,
 \begin{enumerate}
  \item  for $j \ge 2$, $N_i(W_jH^i(\mc{Y}_\infty,\mb{Q})) \subset W_{j-2}H^i(\mc{Y}_\infty,\mb{Q})$ and
  \item the map $N_i^l: \mr{Gr}^W_{i+l} H^i(\mc{Y}_\infty,\mb{Q}) \to \mr{Gr}^W_{i-l} H^i(\mc{Y}_\infty,\mb{Q})$ 
  is an isomorphism for all $l \ge 0$.
   \end{enumerate}
  Now, \cite[Theorem $6.16$]{schvar} states that the induced filtration on  $H^i(\mc{Y}_{\infty},\mb{C})$ defines 
  a mixed Hodge structure $(H^i(\mc{Y}_{\infty},\mb{Z}),W_\bullet,F^\bullet)$.
 \end{rem}
 
 When the central fiber $\mc{Y}_0$ is a reduced simple normal crossings divisor of $\mc{Y}$, 
we have the following description of the specialization morphism.
 \begin{rem}\label{ntor01}
 Suppose that the central fiber $\mc{Y}_0$ is a reduced simple normal crossings divisor of $\mc{Y}$.
 By the local invariant cycle theorem \cite[Theorem $11.43$]{pet}, we have the following exact sequence of mixed Hodge structure:
 \begin{equation}\label{t02}
  H^i(\mc{Y}_0,\mb{Q}) \xrightarrow{\mr{sp}_i} H^i(\mc{Y}_{\infty},\mb{Q}) \xrightarrow{N_i/(2\pi \sqrt{-1})} H^i(\mc{Y}_{\infty},\mb{Q})(-1)
 \end{equation}
 where $\mr{sp}_i$ denotes the specialization morphism.
\end{rem}   
 
 \subsection{Steenbrink's limit mixed Hodge structures}
It is not always easy to compute the monodromy weight filtration defined by Schmid. As a result we will use 
 the Steenbrink spectral sequences below. Note that we do not give the general form of the spectral sequence, instead we 
 restrict to the case relevant to this article.
  
  \begin{prop}[{\cite[Corollaries $11.23$ and $11.41$]{pet} and \cite[Example $3.5$]{ste1}}]\label{tor26}
  Suppose $\mc{Y}$ is regular and $\mc{Y}_0$ is a reduced simple normal crossings divisor of $\mc{Y}$, 
  consisting of exactly two irreducible components, say   $Y_1$ and $Y_2$.
 The \emph{limit weight spectral sequence} $^{^{\infty}}_{_W}E_1^{p,q} \Rightarrow H^{p+q}(\mc{Y}_\infty , \mb{Q})$ consists of the following terms:
 \begin{enumerate}
  \item if $|p| \ge 2$, then $^{^{\infty}}_{_W}E_1^{p,q}=0$,
  \item $^{^{\infty}}_{_W}E_1^{1,q}=H^q(Y_1 \cap Y_2,\mb{Q})(0)$, $^{^{\infty}}_{_W}E_1^{0,q}=H^q(Y_1,\mb{Q})(0) \oplus H^q(Y_2,\mb{Q})(0)$  and $^{^{\infty}}_{_W}E_1^{-1,q}=H^{q-2}(Y_1 \cap Y_2,\mb{Q})(-1)$,
  \item the differential map $d_1:\, ^{^{\infty}}_{_W}E_1^{0,q} \to \, ^{^{\infty}}_{_W}E_1^{1,q}$ is the restriction morphism and \[d_1:\, ^{^{\infty}}_{_W}E_1^{-1,q} \to \, ^{^{\infty}}_{_W}E_1^{0,q}\] is the Gysin morphism.
 \end{enumerate}
The limit weight spectral sequence $^{^{\infty}}_{_W}E_1^{p,q}$ degenerates at $E_2$
and the induced filtration on $H^{p+q}(\mc{Y}_\infty, \mb{Q})$ coincides with the monodromy weight filtration as in Remark \ref{tor25} above.
Similarly, the \emph{weight spectral sequence } $_{_W}E_1^{p,q} \Rightarrow H^{p+q}(\mc{Y}_0 , \mb{Q})$  on $\mc{Y}_0$ consists of the following terms:
 \begin{enumerate}
  \item for $p \ge 2$ or $p<0$, we have $_{_W}E_1^{p,q}=0$,
  \item $_{_W}E_1^{1,q}=H^q(Y_1 \cap Y_2,\mb{Q})(0)$ and $_{_W}E_1^{0,q}=H^q(Y_1,\mb{Q})(0) \oplus H^q(Y_2,\mb{Q})(0)$,
  \item the differential map $d_1:\, _{_W}E_1^{0,q} \to \, _{_W}E_1^{1,q}$ is the restriction morphism.
 \end{enumerate}
 The spectral sequence $_{_W}E_1^{p,q}$ degenerates at $E_2$ and induces a weight filtration on $H^{p+q}(\mc{Y}_0,\mb{Q})$.
 \end{prop}
 
 We note that the resulting weight filtrations on $H^{p+q}(\mc{Y}_\infty , \mb{Q})$ and $H^{p+q}(\mc{Y}_0, \mb{Q})$ are given by:
 \[^{^{\infty}}_{_W}E_2^{p,q}=\mr{Gr}^W_q H^{p+q}(\mc{Y}_\infty , \mb{Q}) \mbox{ and } _{_W}E_2^{p,q}=\mr{Gr}^W_qH^{p+q}(\mc{Y}_0, \mb{Q}).\]
 
 \begin{cor}\label{ntor03}
  Let $\mc{Y}$ and $\mc{Y}_0$ be as in Proposition \ref{tor26}. Then, we have the following exact sequence of mixed Hodge structures:
  \begin{equation}\label{ntor02}
 H^{i-2}(Y_1 \cap Y_2, \mb{Q})(-1) \xrightarrow{f_i} H^i(\mc{Y}_0,\mb{Q}) \xrightarrow{\mr{sp}_i} H^i(\mc{Y}_\infty, \mb{Q}) \xrightarrow{g_i} \mr{Gr}^W_{i+1} H^i(\mc{Y}_\infty, \mb{Q}) \to 0,
  \end{equation}
where $f_i$ is the natural morphism induced by the Gysin morphism from $H^{i-2}(Y_1 \cap Y_2, \mb{Q})(-1)$ to $H^i(Y_1,\mb{Q}) \oplus H^i(Y_2, \mb{Q})$
(use the Mayer-Vietoris sequence associated to $Y_1 \cup Y_2$),
$\mr{sp}_i$ is the \emph{specialization morphism} (see \cite[Theorem $11.29$]{pet})
and $g_i$ is the natural projection.
 \end{cor}

 \begin{proof}
  The corollary is an immediate consequence of Proposition \ref{tor26}.
 \end{proof}

 \subsection{The curve case}
 Consider the flat family $\rho:\widetilde{\mc{X}} \to \Delta$ of projective curves with $\widetilde{\mc{X}}$ regular, 
  $\widetilde{\mc{X}}_t$ smooth of genus $g$ for all $t \in \Delta^*$ and $\widetilde{\mc{X}}_0=Y_1 \cup Y_2$ with $Y_1 \cong \mb{P}^1$, 
  $Y_2$ smooth, irreducible and intersecting $Y_1$ transversally at two points, say $y_1, y_2$. 
 We compute the limit mixed Hodge structure associated to this family of curves.
  This description will be used later in the article
  to give the generators of the weight filtration on the cohomology ring of the moduli space of semi-stable
  sheaves with fixed determinant over an irreducible
  nodal curve.

 \begin{thm}\label{ntor11}
Denote by $f' \in H^2(Y_2,\mb{Z})$, the Poincar\'{e} dual of the fundamental class of $Y_2$ and 
  \[\mr{sp}_2:H^2(Y_1,\mb{Z}) \oplus H^2(Y_2, \mb{Z}) \cong H^2(\widetilde{\mc{X}}_0, \mb{Z}) \xrightarrow{\mr{sp}_2} H^2(\widetilde{\mc{X}}_\infty, \mb{Z})\]
  the specialization morphism as in Corollary \ref{ntor03} composed with the isomorphism arising from the Mayer-Vietoris sequence.
  Then, there exists a basis $e_1, e_2, ..., e_{2g}$ of 
  $H^1(\widetilde{\mc{X}}_\infty, \mb{Z})$ such that 
  \begin{enumerate}
   \item $e_g$ (resp. $e_{2g}$) generates $\mr{Gr}^W_0H^1(\widetilde{\mc{X}}_\infty, \mb{Q})$ (resp. $\mr{Gr}^W_2H^1(\widetilde{\mc{X}}_\infty, \mb{Q})$),
   \item $e_1, e_2, ..., e_{g-1}, e_{g+1}, e_{g+2},..., e_{2g-1}$ form a basis of $\mr{Gr}^W_1H^1(\widetilde{\mc{X}}_\infty, \mb{Q})$,
   \item $e_i \cup e_{i+g} = \mr{sp}_2(0 \oplus -f')$ and $e_i \cup e_j=0$ for $|j-i| \not= g$.
  \end{enumerate}
 \end{thm}

 Before proving the theorem, we note
 that when we say ``$e_{i_1},...,e_{i_r}$ generate $\mr{Gr}^W_j H^1(\widetilde{\mc{X}}_\infty, \mb{Q})$'' we always mean that the image of 
 $e_{i_1},...,e_{i_r}$ in $\mr{Gr}^W_j H^1(\widetilde{\mc{X}}_\infty, \mb{Q})$ (under the natural projection morphism) generate it.
 
 \begin{proof}
 The Mayer-Vietoris sequence associated to the central fiber $\widetilde{\mc{X}}_0$ is
{\small \[0 \to H^0(\widetilde{\mc{X}}_0, \mb{Z}) \to H^0(Y_1, \mb{Z}) \oplus H^0(Y_2, \mb{Z}) \to H^0(Y_1 \cap Y_2, \mb{Z}) \to H^1(\widetilde{\mc{X}}_0, \mb{Z}) \to H^1(Y_1, \mb{Z}) \oplus H^1(Y_2, \mb{Z}) \to 0.\]}
Since $H^1(Y_1,\mb{Z})=0$, this gives us the short exact sequence:
 \begin{equation}\label{ner10}
 0 \to \mb{Z} \xrightarrow{p} H^1(\widetilde{\mc{X}}_0, \mb{Z}) \xrightarrow{q} H^1(Y_2,\mb{Z}) \to 0,
\end{equation}
 inducing isomorphisms $\mb{Q} \stackrel{p}{\cong} \mr{Gr}^W_0H^1(\widetilde{\mc{X}}_0, \mb{Q})$ and $\mr{Gr}^W_1H^1(\widetilde{\mc{X}}_0, \mb{Q}) \stackrel{q}{\cong} H^1(Y_2,\mb{Q})$.
 Since $\rho$ is a flat family of projective curves with $\rho^{-1}(t)$ of genus $g$ for $g \in \Delta^*$, we have 
  \[g=\rho_a(Y_1 \cup Y_2)=\rho_a(Y_1)+\rho_a(Y_2)+Y_1.Y_2-1=\rho_a(Y_2)+1,\]
  where $\rho_a$ denotes the arithmetic genus (use $\rho_a(Y_1)=0$). In other words, $\rho_a(Y_2)=g-1$.
  There exists a symplectic basis $e'_1,  e'_2, ..., e'_{g-1}, e'_{g+1}, e'_{g+2},..., e'_{2g-1}$ of $H^1(Y_2, \mb{Z})$
  such that $e'_i \cup e'_{i+g}=-f'$ and $e'_i \cup e'_j=0$ for $|i-j| \not= g$, where $f' \in H^2(Y_2, \mb{Z})$ is the dual of 
  fundamental class of $Y_2$ (see \cite[\S $1.2$]{carl}). Let $e_i'' \in H^1(\widetilde{\mc{X}}_0, \mb{Z})$ such that 
  $q(e_i'')=e_i'$. Denote by $e_i:=\mr{sp}_1(e_i'')$, where  
  \[\mr{sp}_1: H^1(\widetilde{\mc{X}}_0, \mb{Z}) \to H^1(\widetilde{\mc{X}}_\infty, \mb{Z}),\]
 is the specialization morphism. Since $\mr{sp}_1$ maps  $\mr{Gr}^W_1H^1(\widetilde{\mc{X}}_0, \mb{Z})$
 isomorphically to $\mr{Gr}^W_1H^1(\widetilde{\mc{X}}_\infty, \mb{Z})$ (Corollary \ref{ntor03}), we conclude that 
 $e_1, e_2, ..., e_{g-1}, e_{g+1}, e_{g+2},..., e_{2g-1}$ is a basis of $\mr{Gr}^W_1H^1(\widetilde{\mc{X}}_\infty, \mb{Q})$.
 
 Denote by $i_1:Y_1 \hookrightarrow \widetilde{\mc{X}}_0$ and $i_2:Y_2 \hookrightarrow \widetilde{\mc{X}}_0$ the natural inclusions.
 Since cup-product commutes with pull-back, we have $i_1^*(e_i'' \cup e_j'')=0$ for all $i, j$ (use $H^1(\mb{P}^1)=0$), $i_2^*(e_i'' \cup e_{i+g}'')=e'_i \cup e'_{i+g}=-f'$ and
 $i_2^*(e_i'' \cup e_j'')=e'_i \cup e'_j=0$ for any $|j-i| \not= g$. This implies $e_i \cup e_j=0$ for $|j-i| \not= g$ and $e_i \cup e_{i+g}=\mr{sp}_2(0 \oplus -f')$ under the morphism
 \begin{equation}\label{eq:bet09}
 H^2(Y_1, \mb{Z}) \oplus H^2(Y_2, \mb{Z}) \xleftarrow[\sim]{(i_1^*, i_2^*)} H^2(\widetilde{\mc{X}}_0, \mb{Z}) \xrightarrow{\mr{sp}_2} H^2(\widetilde{\mc{X}}_\infty, \mb{Z}),
 \end{equation}
 where the first isomorphism follows from the Mayer-Vietoris sequence.

 Denote by $e_g:=\mr{sp}_1 \circ p(1)$. Note that $e_g$ generates $W_0H^1(\widetilde{\mc{X}}_\infty,\mb{Q})$ (use Corollary \ref{ntor03}).
 Since the cup-product morphism 
\[H^1(\widetilde{\mc{X}}_\infty, \mb{Z}) \otimes H^1(\widetilde{\mc{X}}_\infty, \mb{Z}) \to H^2(\widetilde{\mc{X}}_\infty, \mb{Z})\]
is a morphism of mixed Hodge structures and $H^2(\widetilde{\mc{X}}_\infty, \mb{Q})$ is pure, we conclude that 
$e_g \cup e_i=0$ for all $1 \le i \le {2g-1}$.
Choose $e_{2g} \in H^1(\widetilde{\mc{X}}_\infty, \mb{Z})$ such that $e_1, e_2,...,e_{2g}$ generates $H^1(\widetilde{\mc{X}}_\infty, \mb{Z})$.
Let $e_i \cup e_{2g} = a_i\mr{sp}_2(0 \oplus -f')$. 
As cup-product is skew-symmetric, we have $a_{2g}=0$.
Moreover, as the cup-product morphism is 
 a perfect pairing, we have $|a_g|=1$. Replace $e_{2g}$ by 
\[\frac{1}{a_g}(e_{2g}-\sum\limits_{i<g} a_ie_{i+g}+\sum\limits_{i>g}^{2g-1} a_ie_{i-g}).\]
Note that $e_{2g} \in H^1(\widetilde{\mc{X}}_\infty, \mb{Z})$, generates $\mr{Gr}^W_2H^1(\widetilde{\mc{X}}_\infty, \mb{Q})$, $e_i \cup e_{2g}=0$ for $i \not= g$ and $e_g \cup e_{2g}=\mr{sp}_2(0 \oplus -f')$.
 This proves the theorem.
 \end{proof}

\section{Relative Gieseker moduli space and Mumford-Newstead isomorphism}\label{sec2a}

In this section we recall the relative Gieseker moduli space and review results on relative Mumford-Newstead isomorphisms as described in \cite[\S $4$]{indpre}.

 \begin{note}\label{tor33}
   Let $X_0$ be an irreducible nodal curve of genus $g \ge 2$, with exactly
 one node, say at $x_0$. Denote by $\pi_0:\widetilde{X}_0 \to X_0$ be the normalization map.
 Since the moduli space of stable curve
 is complete, there exists a regular, flat family of projective curves
  $\pi_1:\mc{X} \to \Delta$ smooth over 
   $\Delta^*$ and central fiber isomorphic to $X_0$ (see \cite[Theorem B$.2$]{bake}).
    Fix an invertible sheaf $\mc{L}$ on $\mc{X}$ of relative odd degree, say $d$. Set $\mc{L}_0:=\mc{L}|_{X_0}$, 
   the restriction of $\mc{L}$ to the central fiber.
   Denote by $\widetilde{\mc{L}}_0:=\pi_0^*(\mc{L}_0)$. Denote by $\mc{L}_s:=\mc{L}|_{\mc{X}_s}$
   for any $s \in \Delta$. For $s \in \Delta^*$, denote by $M_{\mc{X}_s}(2,\mc{L}_s)$
   the moduli space of rank $2$, semi-stable sheaves with determinant $\mc{L}_s$ over $\mc{X}_s$.
   By \cite[Corollary $4.5.5$]{huy}, $M_{\mc{X}_s}(2,\mc{L}_s)$ is non-singular for every 
   $s \in \Delta^*$.
      \end{note}

   \subsection{Relative Gieseker moduli space}\label{subsec1}
There exists a regular, flat, projective family   
  \[\pi_2:\mc{G}(2,\mc{L}) \to \Delta\]
  called the \emph{relative Gieseker moduli space of rank $2$ semi-stable sheaves on $\mc{X}$ with determinant 
  $\mc{L}$}, such that for all $s \in \Delta^*$, $\mc{G}(2,\mc{L})_s:=\pi_2^{-1}(s)=M_{\mc{X}_s}(2,\mc{L}_s)$ and the central fiber $\pi_2^{-1}(0)$, denoted 
  $\mc{G}_{X_0}(2,\mc{L}_0)$, is a reduced simple normal crossings divisor of $\mc{G}(2,\mc{L})$. See \cite[Theorem $2$]{sun1} and \cite[\S $6$]{tha} for the definition and 
  construction of the moduli space $\mc{G}(2,\mc{L})$.
  In particular, $\mc{G}(2,\mc{L})$ is smooth over $\Delta^*$ and satisfies the conditions of Proposition \ref{tor26}.
  
   Denote by $M_{\widetilde{X}_0}(2,\widetilde{\mc{L}}_0)$ the fine moduli space of semi-stable sheaves of rank $2$ and with determinant $\widetilde{\mc{L}}_0$ over $\widetilde{X}_0$ (see \cite[Theorem $4.3.7$ and $4.6.6$]{huy}).
  By \cite[\S $6$]{tha}, $\mc{G}_{X_0}(2,\mc{L}_0)$ can be written as the union of two irreducible components, say $\mc{G}_0$ and $\mc{G}_1$, and 
  $\mc{G}_1$ (resp. $\mc{G}_0 \cap \mc{G}_1$) is isomorphic to a $\p3$ (resp. $\mb{P}^1 \times \mb{P}^1$)-bundle over $M_{\widetilde{X}_0}(2,\widetilde{\mc{L}}_0)$.

 \subsection{Mumford-Newstead isomorphism in families}
 Let us consider the relative version of the construction in \cite{mumn}.
  Denote by  \[\mc{W}:=\mc{X}_{\Delta^*} \times_{\Delta^*} \mc{G}(2,\mc{L})_{\Delta^*}\, \mbox{ and }\, \pi_3: \mc{W} \to \Delta^*\] the natural morphism. Recall, for 
  all $t \in \Delta^*$, the fiber \[\mc{W}_t:=\pi_3^{-1}(t)=\mc{X}_t \times \mc{G}(2,\mc{L})_t \cong \mc{X}_t \times M_{\mc{X}_t}(2,\mc{L}_t).\] 
   There exists a (relative) universal bundle $\mc{U}$ over $\mc{W}$
  associated to the (relative) moduli space $\mc{G}(2,\mc{L})_{\Delta^*}$ i.e., $\mc{U}$ is a vector bundle over $\mc{W}$ such that for each $t \in \Delta^*$,
  $\mc{U}|_{\mc{W}_t}$ is the universal bundle over $\mc{X}_t \times M_{\mc{X}_t}(2,\mc{L}_t)$ associated to fine moduli space $M_{\mc{X}_t}(2,\mc{L}_t)$
  (use \cite[Theorem $9.1.1$]{pan}).
  Denote by $\mb{H}^4_{\mc{W}}:=R^4 \pi_{3_*} \mb{Z}_{\mc{W}}$ the local system associated to $\mc{W}$.
  Using K\"{u}nneth decomposition, we have (see \S \ref{nsec1} for notations)
  \begin{equation}
   \mb{H}^4_{\mc{W}}= \bigoplus\limits_i \left(\mb{H}^i_{{\mc{X}}_{\Delta^*}} \otimes \mb{H}^{4-i}_{\mc{G}(2,\mc{L})_{\Delta^*}}\right).
  \end{equation}
  Denote by $c_2(\mc{U})^{1,3} \in \Gamma\left(\mb{H}^1_{{\mc{X}}_{\Delta^*}} \otimes \mb{H}^{3}_{\mc{G}(2,\mc{L})_{\Delta^*}}\right)$ 
  the image of the second Chern class $c_2(\mc{U}) \in \Gamma(\mb{H}^4_{\mc{W}})$ under the natural projection from 
  $\mb{H}^4_{\mc{W}}$ to $\mb{H}^1_{{\mc{X}}_{\Delta^*}} \otimes \mb{H}^{3}_{\mc{G}(2,\mc{L})_{\Delta^*}}$.
  Using Poincar\'{e} duality applied to the local system $\mb{H}^1_{{\mc{X}}_{\Delta^*}}$ (see \cite[\S I.$2.6$]{kuli}), we have
\begin{equation}\label{ner08}
 \mb{H}^1_{{\mc{X}}_{\Delta^*}} \otimes \mb{H}^{3}_{\mc{G}(2,\mc{L})_{\Delta^*}} \stackrel{\mr{PD}}{\cong}\left(\mb{H}^1_{{\mc{X}}_{\Delta^*}}\right)^\vee \otimes \mb{H}^{3}_{\mc{G}(2,\mc{L})_{\Delta^*}} \cong \Hc\left(\mb{H}^1_{{\mc{X}}_{\Delta^*}}, \mb{H}^{3}_{\mc{G}(2,\mc{L})_{\Delta^*}}\right).
\end{equation}
  Therefore, $c_2(\mc{U})^{1,3}$ induces a homomorphism $\Phi_{\Delta^*}: \mb{H}^1_{\mc{X}_{\Delta^*}} \to \mb{H}^{3}_{\mc{G}(2,\mc{L})_{\Delta^*}}.$
  
  By \cite[Lemma $1$ and Proposition $1$]{mumn}, we conclude that the homomorphism $\Phi_{\Delta^*}$
  is an isomorphism such that the induced 
  isomorphism on the associated vector bundles:
  \[\Phi_{\Delta^*}:\mc{H}^1_{\mc{X}_{\Delta^*}} \xrightarrow{\sim} \mc{H}^{3}_{\mc{G}(2,\mc{L})_{\Delta^*}} \mbox{ satisfies } \Phi_{\Delta^*}(F^p\mc{H}^1_{\mc{X}_{\Delta^*}})= F^{p+1}\mc{H}^{3}_{\mc{G}(2,\mc{L})_{\Delta^*}} \mbox{ for all }p \ge 0.\]

 \subsection{Limit Mumford-Newstead isomorphism} 
 We now extend the isomorphism $\Phi_{\Delta^*}$ to the entire disc $\Delta$ and show that the induced morphism 
 on the central fibers is an isomorphism of limit mixed Hodge structures. We do this using the monodromy operator (see \S \ref{sec3}).
 In order to guarantee that the monodromy operator is unipotent, we want the central fibers of the relevant families of projective varieties
 to be reduced simple, normal crossings divisors. 
 The family $\pi_2$ of moduli spaces already satisfies this criterion.
 Unfortunately, the central fiber of $\pi_1$ is $X_0$, which is not a simple normal crossings divisor. 
 We can easily rectify this problem by blowing up $\mc{X}$
  at the point $x_0$.
  Denote by $\widetilde{\mc{X}}:=\mr{Bl}_{x_0}\mc{X}$ and by 
  \begin{equation}\label{eq:tor01}
   \widetilde{\pi}_1:\widetilde{\mc{X}} \to \mc{X} \xrightarrow{\pi_1} \Delta.
  \end{equation}
 Note that the central fiber of $\widetilde{\pi}_1$ is the union of two irreducible components, the normalization $\widetilde{X}_0$ of $X_0$ and 
  the exceptional divisor $F \cong \mb{P}^1_{x_0}$ intersecting $\widetilde{X}_0$ at the two points 
  over $x_0$.
 
 Let $\ov{\mc{H}}^1_{\mc{X}_{\Delta^*}}$ and $\ov{\mc{H}}^{3}_{\mc{G}(2,\mc{L})_{\Delta^*}}$ be the canonical extensions of $\mc{H}^1_{\mc{X}_{\Delta^*}}$
  and $\mc{H}^{3}_{\mc{G}(2,\mc{L})_{\Delta^*}}$, respectively. Then, the morphism $\Phi_{\Delta^*}$ extend to the entire disc:
  \[\widetilde{\Phi}:\ov{\mc{H}}_{\widetilde{\mc{X}}}^1 \xrightarrow{\sim} \ov{\mc{H}}^3_{\mc{G}(2,\mc{L})}.\]
  Using the identification \eqref{tor23} and restricting $\widetilde{\Phi}$ to the central fiber, we have an isomorphism:
  \begin{equation}\label{ner11}
   \widetilde{\Phi}_0 : H^1(\widetilde{\mc{X}}_\infty,\mb{Q}) \xrightarrow{\sim} H^3(\mc{G}(2,\mc{L})_\infty,\mb{Q}).
  \end{equation}
   Recall, $\widetilde{\Phi}_0$ is an isomorphism of mixed Hodge structures:
   
  \begin{thm}\label{ner01}
   For the extended morphism $\widetilde{\Phi}$, we have 
   $\widetilde{\Phi}(F^p\ov{\mc{H}}_{\widetilde{\mc{X}}}^1)= F^{p+1}\ov{\mc{H}}^3_{\mc{G}(2,\mc{L})}$  for  $p=0,1$  and  
   $\widetilde{\Phi}(\ov{\mb{H}}_{\widetilde{\mc{X}}}^1)=\ov{\mb{H}}^3_{\mc{G}(2,\mc{L})}$.
   Moreover, $\widetilde{\Phi}_0(W_iH^1(\widetilde{\mc{X}}_\infty,\mb{Q}))=W_{i+2}H^3(\mc{G}(2,\mc{L})_\infty,\mb{Q})$
   for all $i \ge 0$.
  \end{thm}

  \begin{proof}
   See \cite[Proposition $4.1$]{indpre} for a proof of the statement.
 \end{proof}

 \section{Computing the kernel of the Gysin morphism}\label{sec2b}
 
 Notations as in Notation \ref{tor33} and \S \ref{subsec1}.
 The goal of this section is to compute the kernel of the Gysin morphism $f_i$ as in \eqref{ntor02}, in the case when the flat family $\rho$
is the relative Gieseker moduli space of rank $2$ semi-stable sheaves with fixed determinant associated to a degenerating family of 
smooth curves (Theorem \ref{bet01}).
 
 \subsection{Cohomology of the fibers of $P_0, \mc{G}_1$ and $\mc{G}_0 \cap \mc{G}_1$}\label{note:ner01}
 
 Recall the \emph{wonderful compactification} $\ov{\mr{SL}}_2 \subset \mb{P}^4 \cong \mb{P}(\mr{End}(\mb{C}^2) \oplus \mb{C})$ of $\mr{SL}_2$ consisting of points $[x_0:x_1:...:x_4]$ such that 
$x_0x_3-x_1x_2=x_4^2$ (see \cite[Definition $3.3.1$]{pezz} for the general definition of 
wonderful compactification). 
Denote by $j_1:\ov{\mr{SL}}_2 \hookrightarrow \mb{P}^4$ the natural inclusion 
as a quadric hypersurface.
 Recall that $\mc{G}_1$ (resp. $\mc{G}_0 \cap \mc{G}_1$) is a $\p3$ (resp. $\mb{P}^1 \times \mb{P}^1$)-bundle over $M_{\widetilde{X}_0}(2,\widetilde{\mc{L}}_0)$.
 Denote by 
\[\rho_1: \mc{G}_0 \cap \mc{G}_1 \longrightarrow M_{\widetilde{X}_0}(2,\widetilde{\mc{L}}_0) \, 
\mbox{ and } \rho_2: \mc{G}_1 \longrightarrow M_{\widetilde{X}_0}(2,\widetilde{\mc{L}}_0)\]
 the natural projections. 
 Recall by \cite[\S $6$]{tha} that there exists an $\ov{\mr{SL}}_2$-bundle $P_0$ over $M_{\widetilde{X}_0}(2,\widetilde{\mc{L}}_0)$
 and natural inclusions
  \[i_1: \mc{G}_0 \cap \mc{G}_1 \hookrightarrow P_0, \,  i_2: \mc{G}_0 \cap \mc{G}_1 \hookrightarrow \mc{G}_1
  \mbox{ and } i_3:\mc{G}_0 \cap \mc{G}_1 \hookrightarrow \mc{G}_0\]
  such that for the natural projection $\rho_3: P_0 \longrightarrow M_{\widetilde{X}_0}(2,\widetilde{\mc{L}}_0)$, we have
for any $y \in M_{\widetilde{X}_0}(2,\widetilde{\mc{L}}_0)$, identifying the fiber $\rho_1^{-1}(y)$
(resp. $\rho_2^{-1}(y)$, $\rho_3^{-1}(y)$) with $\mb{P}^1 \times \mb{P}^1$ (resp. $\p3, \ov{\mr{SL}}_2$), the natural inclusions 
$i_{1,y}$ and $i_{2,y}$, induced by $i_1$ and $i_2$ respectively, sit
in the 
following diagram:
\[\begin{diagram}
 \mb{P}^1 \times \mb{P}^1&\rTo^{i_{1,y}}&\ov{\mr{SL}}_2\\
 \dTo^{i_{2,y}}&\circlearrowleft&\dTo^{j_1}\\
  \p3&\rTo^{j_2}&\mb{P}^4
  \end{diagram}\]
where $j_2([x_0:...:x_3])=[x_0:...:x_3:0]$ and $i_{1,y}$ (resp. $i_{2,y}$) is the Segre embedding sending 
\[ [s:t] \times [u:v] \in \mb{P}^1 \times \mb{P}^1 \mbox{ to } [su:sv:tu:tv:0] \in \ov{\mr{SL}}_2\, \, (\mbox{resp. } [su:sv:tu:tv] \in \p3).\]

 Let $\xi_0$ be a generator of $H^2(\mb{P}^1, \mb{Q})$, $\pr_i$ the natural 
projections from $\mb{P}^1 \times \mb{P}^1$ to $\mb{P}^1$ and $\xi_i:=\pr_i^*(\xi_0)$.
Using the Künneth decomposition, we have 
\[H^1(\mb{P}^1 \times \mb{P}^1, \mb{Q})=0=H^3(\mb{P}^1 \times \mb{P}^1, \mb{Q}), H^0(\mb{P}^1 \times \mb{P}^1, \mb{Q})= \mb{Q},    
H^2(\mb{P}^1 \times \mb{P}^1, \mb{Q})= \mb{Q}\xi_1 \oplus \mb{Q}\xi_2\mbox{ and }\]
$H^4(\mb{P}^1 \times \mb{P}^1, \mb{Q})=\mb{Q} \xi_1.\xi_2$.
 Since $\ov{\mr{SL}}_2$ is a 
 quadric hypersurface in $\mb{P}^4$, the Lefschetz hyperplane section theorem implies that
  $H^{2i}(\ov{\mr{SL}}_2,\mb{Q}) \cong \mb{Q}$ for $0 \le i \le 3$ and 
  $H^1(\ov{\mr{SL}}_2,\mb{Q})=0=H^5(\ov{\mr{SL}}_2,\mb{Q})$.
 It is also known that $H^3(\ov{\mr{SL}}_2,\mb{Q})=0$.
 Denote by $\xi \in H^2(\mb{P}^4, \mb{Z})$ a generator, $\xi':=j_1^*(\xi)$ and $\xi'':=j_2^*(\xi)$.

 \subsection{Kernel of the Gysin morphisms}
 We now compute the kernel of the Gysin morphisms $i_{2,*}$ and $i_{3,*}$. The first step is to 
 determine the kernel of $i_{1,*}$ and $i_{2,*}$ (Proposition \ref{ner15}). This is done using the Leray-Hirsch theorem.

 \begin{prop}\label{ner15}
 We have, 
   \[\ker((i_{1,*},i_{2,*}):H^{i-2}(\mc{G}_0 \cap \mc{G}_1,\mb{Q})(-1) \to H^i(P_0,\mb{Q}) \oplus H^i(\mc{G}_1,\mb{Q})) \cong H^{i-4}(M_{\widetilde{X}_0}(2,\widetilde{\mc{L}}_0), \mb{Q})(\xi_1 \oplus -\xi_2).\]
   \end{prop}

   \begin{proof}
It is easy to check that,
\begin{align}\label{ner13}
 \ker(((i_{1,y})_*, (i_{2,y})_*): H^{2j}(\mb{P}^1 \times \mb{P}^1, \mb{Q}) \to H^{2j+2}(\ov{\mr{SL}}_2, \mb{Q}) \oplus H^{2j+2}(\p3, \mb{Q}))\, & \, \cong \mb{Q} \, \mbox{ if } j=1\\
 \, & \, = 0\, \mbox{ if } j \not= 1.\label{ner14}
\end{align}
By the Deligne-Blanchard theorem \cite{delibla} (the Leray spectral sequence
degenerates at $E_2$ for smooth families), we have 
{\small \[H^i(\mc{G}_0 \cap \mc{G}_1, \mb{Q}) \cong  \oplus_{_j} H^{i-j}(R^j \rho_{1,*}\mb{Q}),
 H^i(\mc{G}_1,\mb{Q}) \cong  \oplus_{_j} H^{i-j}(R^j \rho_{2,*}\mb{Q}) \mbox{ and } H^i(P_0,\mb{Q}) \cong \oplus_{_j} H^{i-j}(R^j \rho_{3,*}\mb{Q}).
\]}
Since $M_{\widetilde{X}_0}(2,\widetilde{\mc{L}}_0)$ is smooth and simply connected, the local systems
$R^j \rho_{1,*} \mb{Q}$, $R^j \rho_{2,*} \mb{Q}$ and $R^j \rho_{3,*} \mb{Q}$ are trivial.
Therefore, for any $y \in M_{\widetilde{X}_0}(2,\widetilde{\mc{L}}_0)$, the natural morphisms
\begin{align*}
  & H^i(\mc{G}_0 \cap \mc{G}_1, \mb{Q}) \twoheadrightarrow H^0(R^i \rho_{1,*} \mb{Q}) \to H^i((\mc{G}_0 \cap \mc{G}_1)_y,\mb{Q}),\\
  & H^i(\mc{G}_1, \mb{Q}) \twoheadrightarrow H^0(R^i \rho_{2,*} \mb{Q}) \to H^i(\mc{G}_{1,y},\mb{Q}) \mbox{ and }
   H^i(P_0, \mb{Q}) \twoheadrightarrow H^0(R^i \rho_{3,*} \mb{Q}) \to H^i(P_{0,y},\mb{Q}),
\end{align*}
are surjective. 
Then, Leray-Hirsch theorem (see \cite[Theorem $7.33$]{v4}), implies that for any closed point $y \in M_{\widetilde{X}_0}(2,\widetilde{\mc{L}}_0)$,
{\small \begin{align}
 H^i(\mc{G}_0 \cap \mc{G}_1, \mb{Q}) & \cong \bigoplus_{j \ge 0} (H^j((\mc{G}_0 \cap \mc{G}_1)_y,\, \mb{Q}) \otimes H^{i-j}(M_{\widetilde{X}_0}(2,\widetilde{\mc{L}}_0))) \nonumber \\
& \cong H^i(M_{\widetilde{X}_0}(2,\widetilde{\mc{L}}_0)) \oplus H^{i-2}(M_{\widetilde{X}_0}(2,\widetilde{\mc{L}}_0)) \otimes (\mb{Q}\xi_1 \oplus \mb{Q}\xi_2) \oplus H^{i-4}(M_{\widetilde{X}_0}(2,\widetilde{\mc{L}}_0)) \xi_1\xi_2, \label{ntor04}\\ 
H^i(\mc{G}_1, \mb{Q}) & \cong \bigoplus_{j \ge 0} (H^j(\mc{G}_{1,y},\, \mb{Q}) \otimes H^{i-j}(M_{\widetilde{X}_0}(2,\widetilde{\mc{L}}_0)))
 \cong \bigoplus_{j \ge 0} H^{i-2j}(M_{\widetilde{X}_0}(2,\widetilde{\mc{L}}_0)) \otimes (\xi'')^j \label{ntor05}\\
 H^i(P_0, \mb{Q}) & \cong \bigoplus_{j \ge 0} (H^j(P_{0,y},\, \mb{Q}) \otimes H^{i-j}(M_{\widetilde{X}_0}(2,\widetilde{\mc{L}}_0)))
 \cong \bigoplus_{j \ge 0} H^{i-2j}(M_{\widetilde{X}_0}(2,\widetilde{\mc{L}}_0)) \otimes (\xi')^j \label{ntor06}
\end{align}}
 Using the projection formula (see \cite[Lemma B$.26$]{pet}) and  the identifications described in the 
 Leray-Hirsch theorem (identifying certain cohomology classes with 
 their restrictions to $y$), we have for any 
 $\alpha \in H^{j-2}((\mc{G}_0 \cap \mc{G}_1)_y, \mb{Q})$ and $\beta \in H^{i-j-2}(M_{\widetilde{X}_0}(2,\widetilde{\mc{L}}_0), \mb{Q})$, 
 \[i_{1,*}(\alpha \cup \rho_1^*\beta)=i_{1,*}(\alpha \cup i_1^*\rho_3^*\beta)=(i_{1,y})_*\alpha \cup \rho_3^*\beta
  \mbox{ and } i_{2,*}(\alpha \cup \rho_1^*\beta)=i_{2,*}(\alpha \cup i_2^*\rho_2^*\beta)=(i_{2,y})_*\alpha \cup \rho_2^*\beta.
 \]
 Note that for $\beta \not= 0$, $(i_{1,y})_*\alpha \cup \rho_3^*\beta=0$ (resp. $(i_{2,y})_*\alpha \cup \rho_2^*\beta=0$) if and only if 
 $(i_{1,y})_*\alpha=0$ (resp. $(i_{2,y})_*\alpha=0$).
Using the decompositions above and \eqref{ner13}, \eqref{ner14}, this implies $\ker((i_{1,*},i_{2,*}))$ is isomorphic to 
$H^{i-4}(M_{\widetilde{X}_0}(2,\widetilde{\mc{L}}_0), \mb{Q})(\xi_1 \oplus -\xi_2)$.
This proves the proposition.
   \end{proof}

 \begin{thm}\label{bet01}
  The kernel of the Gysin morphism \[(i_{3,_*}, i_{2,_*}): H^{i-2}(\mc{G}_0 \cap \mc{G}_1, \mb{Q})(-1) \to H^i(\mc{G}_0,\mb{Q}) \oplus H^i(\mc{G}_1, \mb{Q})\]  is 
  isomorphic to $H^{i-4}(M_{\widetilde{X}_0}(2,\widetilde{\mc{L}}_0),\mb{Q})(\xi_1 \oplus -\xi_2)$. In particular, the kernel of the 
morphism $f_i$ from $H^{i-2}(\mc{G}_0 \cap \mc{G}_1, \mb{Q})(-1)$ to 
  $H^i(\mc{G}_{X_0}(2,\mc{L}_0), \mb{Q})$ induced by the 
  Gysin morphism $(i_{3,_*}, i_{2,_*})$ (use the Mayer-Vietoris
  sequence associated to $\mc{G}_0 \cup \mc{G}_1$) is isomorphic
  to $H^{i-4}(M_{\widetilde{X}_0}(2,\widetilde{\mc{L}}_0),\mb{Q})(\xi_1 \oplus -\xi_2)$.
 \end{thm}

  \begin{proof}
  By \cite[\S $6$]{tha}, there exist closed subschemes $Z \subset P_0$ and $Z' \subset \mc{G}_0$ 
  such that $P_0 \backslash Z \cong \mc{G}_0 \backslash Z'$.
  Using \cite{gies} (see also \cite[P. $27$]{tha} or \cite[Remark $6.5$(c), Theorem $6.2$]{sesh4}), one can observe that 
  $Z \cap \Ima(i_1)=\emptyset=Z' \cap \Ima(i_3)$ 
  and there exists a smooth, projective variety $W$ along with proper, birational morphisms $\tau_1:W \to P_0$ and $\tau_2:W \to \mc{G}_0$ such that 
  \[W \backslash \tau_1^{-1}(Z) \cong P_0 \backslash Z \cong \mc{G}_0 \backslash Z' \cong W \backslash \tau_2^{-1}(Z').\]
 Therefore, there exists a natural closed immersion $l:\mc{G}_0 \cap \mc{G}_1 \to W$ such that $i_1=\tau_1 \circ l$ and $i_3=\tau_2 \circ l$.
 We claim that, given any $\xi \in H^{k-2}(\mc{G}_0 \cap \mc{G}_1,\mb{Q})$, we have $\tau_1^* \circ i_{1,*}(\xi)=l_*(\xi)=\tau_2^* \circ i_{3,*}(\xi)$.
  Indeed, since $\Ima(i_1)$ (resp. $\Ima(i_3)$) does not intersect $Z$ (resp. $Z'$), the pullback of $l_*(\xi)$ to $\tau_1^{-1}(Z)$ and $\tau_2^{-1}(Z')$ vanish.
  Using the (relative) cohomology exact sequence (see \cite[Proposition $5.54$]{pet}), we conclude that there exists $\beta_1 \in H^k(P_0)$ and $\beta_2 \in H^k(\mc{G}_0)$
  such that $\tau_1^*(\beta_1)=l_*(\xi)=\tau_2^*(\beta_2)$. Applying $\tau_{1,*}$ and $\tau_{2,*}$ to the two equalities respectively and using 
  \cite[Proposition B$.27$]{pet}, we get 
  \[\beta_1=\tau_{1,*}\tau_1^*(\beta_1)=\tau_{1,*}l_*(\xi)=i_{1,*}(\xi) \mbox{ and } \beta_2=\tau_{2,*}\tau_2^*(\beta_2)=\tau_{2,*}l_*(\xi)=i_{3,*}(\xi).\]
  In other words, $\tau_1^* \circ i_{1,*}(\xi)=l_*(\xi)=\tau_2^* \circ i_{3,*}(\xi)$. This proves the claim. 
Using \cite[Theorem $5.41$]{pet}, we then conclude that $i_{1,*}(\xi)=0$ (resp. $i_{3,*}(\xi)=0$) if and only if $l_*(\xi)=0$. In other words, 
   $\ker(i_{1,*}) \cong \ker(i_{3,*})$. 
  Using Proposition \ref{ner15}, we have that $\ker(f_i) \cong H^{i-4}(M_{\widetilde{X}_0}(2,\widetilde{\mc{L}}_0),\mb{Q})(\xi_1 \oplus -\xi_2)$.
 This proves the theorem.
   \end{proof}

 \section{On a conjecture of Mumford}\label{sec4}
 In Theorem \ref{th:bet1}, we give a complete set of relations between the generators of the cohomology ring of the moduli space of 
 rank $2$ semi-stable sheaves with fixed determinant over an irreducible nodal curve, analogous to a 
 classical conjecture of Mumford as proved in \cite{kirw}. We use notations as in Notation \ref{tor33} and \S \ref{subsec1}. 
 
 \subsection{Generators of the cohomology ring $H^*(\mc{G}(2,\mc{L})_\infty, \mb{Q})$}\label{subsecgen}
  We first use the classical result of Newstead \cite{new1}
 to determine the generators of the cohomology ring $H^*(\mc{G}(2,\mc{L})_\infty, \mb{Q})$.
 Let $\psi^\infty_i:=\widetilde{\Phi}_0(e_i)$, where $e_i \in H^1(\widetilde{\mc{X}}_\infty, \mb{Z})$ as in Theorem \ref{ntor11}, $1 \le i \le 2g$
and $\widetilde{\Phi}_0$ as in \eqref{ner11}.
    Fix $s \in \Delta^*$.
 Let $\psi_{i,s} \in H^3(\mc{G}(2,\mc{L})_s, \mb{Z})$ the image of $\psi_i^\infty$ under the natural isomorphism 
 \begin{equation}\label{ntor15}
  \eta_j: H^j(\mc{G}(2,\mc{L})_\infty, \mb{Z}) \xrightarrow{\sim} H^j(\mc{G}(2,\mc{L})_s, \mb{Z}), \mbox{ for all } j \ge 0.
 \end{equation}
 Using \cite[Theorem $1$]{new1}, one can observe that for any $s \in \Delta^*$, 
 there exist elements \[\alpha_s \in H^{1,1}(\mc{G}(2,\mc{L})_s, \mb{Z}) \mbox{ and } \beta_s \in 
  H^{2,2}(\mc{G}(2,\mc{L})_s, \mb{Z})\] such that the cohomology ring $H^*(\mc{G}(2,\mc{L})_s, \mb{Q})$   is generated by 
  $\alpha_s, \beta_s, \psi_{1,s}, \psi_{2,s}, ..., \psi_{2g,s}$.
   Let \[\alpha_\infty \in H^2(\mc{G}(2,\mc{L})_\infty, \mb{Z}) \mbox{ and } \beta_\infty \in H^4(\mc{G}(2,\mc{L})_\infty, \mb{Z})\]
 the preimage of $\alpha_s$ and  $\beta_s$, respectively, under the natural isomorphism \eqref{ntor15}.
  It is immediate that the cohomology ring $H^*(\mc{G}(2,\mc{L})_\infty, \mb{Q})$ is generated by 
  $\alpha_\infty, \beta_\infty, \psi_{1}^\infty, \psi_{2}^\infty, ..., \psi_{2g}^\infty$.

  \subsection{Relations on the cohomology rings $H^*(M_{\widetilde{X}_0}(2, \widetilde{\mc{L}}_0), \mb{Q})$ and $H^*(\mc{G}(2,\mc{L})_s, \mb{Q})$}\label{subsecrel}
  Denote by $\psi_i:=\Phi_1(\psi_i^{\infty})$ for $1 \le i \le g-1$
   and $\psi_i:=\Phi_1(\psi_{i+1}^{\infty})$ for  $g \le i \le 2g-2$.
   Let $\alpha \in H^2(M_{\widetilde{X}_0}(2,\widetilde{\mc{L}}_0), \mb{Z})$ (resp. $\beta \in H^2(M_{\widetilde{X}_0}(2,\widetilde{\mc{L}}_0), \mb{Z})$) such that $\alpha$ (resp. $\alpha^2$ and $\beta$) 
   generates $H^2(M_{\widetilde{X}_0}(2,\widetilde{\mc{L}}_0), \mb{Q})$
   (resp. $H^4(M_{\widetilde{X}_0}(2,\widetilde{\mc{L}}_0), \mb{Z})$).
 Denote by $\psi_\infty:=\sum_{i=1}^g \psi_i^\infty\psi_{i+g}^\infty$, \[  \psi:=\sum_{i=1}^{g-1} \psi_i\psi_{i+g} \mbox{ and } P_i:=\ker\left(\psi^{g-i}: \bigcup\limits_{j=1}^i H^3(M_{\widetilde{X}_0}(2,\widetilde{\mc{L}}_0), \mb{Q}) \to \bigcup\limits_{j=1}^{2g-i} H^3(M_{\widetilde{X}_0}(2,\widetilde{\mc{L}}_0), \mb{Q})\right).
           \]
Given an ordered set $(i_1,...,i_k)=I$ with $1 \le i_1<i_2<...<i_k \le 2g$, denote by $\psi_I^\infty:=\psi_{i_1}^\infty... \psi_{i_k}^\infty$. 
For $i \le g$, denote by $P_i^{\infty}$ the $\mb{Q}$-vector space generated by elements of the form
\[\sum\limits_I a_I\psi_I^\infty \in \ker\left(\psi^{g-i+1}_\infty: \bigcup\limits_{j=1}^i H^3(\mc{G}(2,\mc{L})_\infty, \mb{Q}) \to \bigcup\limits_{j=1}^{2g-i+2} H^3(\mc{G}(2,\mc{L})_\infty, \mb{Q})\right),\, \, a_I \in \mb{Q},\]
such that for all  $I$, if $2g \in I$  then  $g \in I$.
Define the ideals $I_k \subset \mb{Q}[\alpha,\beta,\psi]$ and 
$I_k^\infty \subset \mb{Q}[\alpha_\infty,\beta_\infty,\psi_\infty]$
generated by $(\zeta_k, \zeta_{k+1}, \zeta_{k+2})$ and $(\zeta_k^\infty, \zeta_{k+1}^\infty, \zeta_{k+2}^\infty)$ respectively, 
  where $\zeta_i$ and $\zeta_i^\infty$ are recursively defined by $\zeta_0=1, \zeta_0^\infty=1, \zeta_i=0=\zeta_i^\infty$ for $i<0$, 
  \[\zeta_{k+1}:=\alpha\zeta_k+k^2\beta\zeta_{k-1}+2k(k-1)\psi\zeta_{k-2} \mbox{ and }
  \zeta_{k+1}^\infty:=\alpha_\infty\zeta_k^\infty+k^2\beta_\infty\zeta_{k-1}^\infty+2k(k-1)\psi_\infty\zeta_{k-2}^\infty.\] 
Using \cite[Theorem $3.2$]{kingn}, the natural morphism $\nu_0$ from $\bigoplus\limits_{k=0}^{g-1} P_k \otimes \mb{Q}[\alpha, \beta, \psi]$
to $H^*(M_{\widetilde{X}_0}(2, \widetilde{\mc{L}}_0), \mb{Q})$ is surjective with kernel $\oplus_k P_k \otimes I_{g-k-1}$.
  Denote by \[\psi_s:=\sum_{i=1}^{g} \psi_{i,s}\psi_{i+g,s} \mbox{ and } P_i^{(s)}:=\ker\left(\psi_s^{g-i+1}: \bigcup\limits_{j=1}^i H^3(\mc{G}(2,\mc{L})_s, \mb{Q}) \to \bigcup\limits_{j=1}^{2g-i+2} H^3(\mc{G}(2,\mc{L})_s, \mb{Q})\right).\]
Similarly as before, the obvious map
\[\nu: \bigoplus\limits_{k=0}^g P_k^{(s)} \otimes \mb{Q}[\alpha_s, \beta_s, \psi_s] \to H^*(\mc{G}(2,\mc{L})_s,\mb{Q})\]
is surjective with kernel isomorphic to $\oplus P_k^{(s)} \otimes I_{g-k,s}$, where $I_{g-k,s}$ is defined identically as $I^\infty_{g-k}$ above,
after replacing $\alpha_\infty, \beta_\infty$ and $\psi_\infty$ by $\alpha_s, \beta_s$ and $\psi_s$, respectively.

\subsection{Comparing pure Hodge structures on $\mbf{\mr{Gr}^W_3 H^3(\mc{G}(2,\mc{L})_\infty, \mb{Q})}$ and $\mbf{H^3(M_{\widetilde{X}_0}(2,\widetilde{\mc{L}}_0), \mb{Q})}$}
  The cohomology ring of $M_{\widetilde{X}_0}(2,\widetilde{\mc{L}}_0)$ will also play an important role in this section and the next. Recall,
  \cite[Proposition $1$]{mumn} states that there exists an isomorphism of pure Hodge structures:
 \[\Phi_0':H^1(\widetilde{X}_0, \mb{Z}) \xrightarrow{\sim} H^3(M_{\widetilde{X}_0}(2,\widetilde{\mc{L}}_0), \mb{Z}).\]
  Using the short exact sequence \eqref{ner10} and Theorem \ref{ner01}, we have the composed morphism 
  \[\Phi_1: \mr{Gr}^W_3 H^3(\mc{G}(2,\mc{L})_\infty, \mb{Q}) \to H^3(M_{\widetilde{X}_0}(2,\widetilde{\mc{L}}_0), \mb{Q})  \mbox{ defined by }\]
    {\small \[\mr{Gr}^W_3 H^3(\mc{G}(2,\mc{L})_\infty, \mb{Q})
    \xleftarrow[\sim]{\widetilde{\Phi}_0}\mr{Gr}^W_1H^1(\widetilde{\mc{X}}_\infty, \mb{Q}) \xleftarrow[\sim]{\mr{sp}_1}  
    \mr{Gr}^W_1H^1(\widetilde{\mc{X}}_0,\mb{Q}) \xrightarrow[\sim]{q}
    H^1(\widetilde{X}_0,\mb{Q}) \xrightarrow[\sim]{\Phi_0'} 
       H^3(M_{\widetilde{X}_0}(2,\widetilde{\mc{L}}_0), \mb{Q}),
  \]}
where the first isomorphism is given by \eqref{ner11}, the second and third isomorphisms are given in the proof of Theorem \ref{ntor11}.
  By Theorem \ref{ner01}, $\widetilde{\Phi}_0$ is an isomorphism of pure Hodge structures. Also, note that the last three morphisms in the 
  composed morphism $\Phi_1$ are morphisms of pure Hodge structures. Therefore, $\Phi_1$ is an isomorphism of pure Hodge structures.
  
  \subsection{Generalized Mumford's conjecture on $H^*(\mc{G}_{X_0}(2,\mc{L}_0), \mb{Q})$}
 We briefly discuss the idea of the proof of Theorem \ref{th:bet1} below. Combining \eqref{ntor02} with 
 Theorem \ref{bet01}, we prove that $\oplus_i \ker(\mr{sp}_i)$ is generated as a polynomial ring over 
 $H^*(M_{\widetilde{X}_0}(2,\widetilde{\mc{L}}_0), \mb{Q})$ by two variables $X$ and $Y$ satisfying $X^2=Y^2=X-Y=0$.
 Then, \S \ref{subsecrel} gives us the relations between the generators of $\oplus_i \ker(\mr{sp}_i)$.
 To obtain the relations between the generators of $\oplus_i \Ima(\mr{sp}_i)$, we use the isomorphism $\Phi_1$ above, 
 along with the description of the generators of $H^i(\mc{G}(2,\mc{L})_s,\mb{Q})$ for all $i \ge 0$, as given in \cite[Remark $5.3$]{kingn}.
 The theorem would then follow immediately.
 
  \begin{thm}\label{th:bet1}
 We have the following isomorphism of \emph{graded rings}:
 \[H^*(\mc{G}_{X_0}(2,\mc{L}_0),\mb{Q}) \cong \left(\bigoplus\limits_i P_i^\infty \otimes \frac{\mb{Q}[\alpha_\infty, \beta_\infty, \psi_\infty]}{I^\infty_{g-i}}\right) \oplus \left(\bigoplus\limits_i P_{i-2} \otimes \frac{\mb{Q}[\alpha, \beta, \psi, X, Y]}{(I_{g-i-3}, X^2, Y^2, X-Y)}\right).\]
 (Note that the multiplicative identity of the ring $H^*(\mc{G}_{X_0}(2,\mc{L}_0),\mb{Q})$ lies in the first summand.)
 \end{thm}

\begin{proof}
 For any $s \in \Delta^*$, consider the natural isomorphism $\phi_s:H^1(\widetilde{\mc{X}}_\infty, \mb{Z}) \xrightarrow{\sim} H^1(\mc{X}_s,\mb{Z})$ (induced by the closed immersion of $\mc{X}_s$ as a fiber of $\widetilde{\mc{X}}_\infty$).
Denote by $e_{i,s}:=\phi_s(e_i)$, where $e_i$ as in Theorem \ref{ntor11}. 
Consider the composed morphism (use \eqref{eq:bet09}):
\[\ov{\phi}_s: \mb{Z}[\mb{P}^1]^{\vee} \oplus \mb{Z}f' = H^2(\mb{P}^1,\mb{Z}) \oplus H^2(\widetilde{X}_0,\mb{Z}) \cong H^2(\widetilde{\mc{X}}_0,\mb{Z}) \xrightarrow{\mr{sp}_2} H^2(\widetilde{\mc{X}}_\infty,\mb{Z})
 \xrightarrow[\sim]{\phi_s} H^2(\mc{X}_s,\mb{Z}),
\]
where $f' \in H^2(\widetilde{X}_0,\mb{Z})$ is
the Poincar\'{e} dual of the 
fundamental class of $\widetilde{X}_0$.  Since $H^2(\widetilde{\mc{X}}_\infty,\mb{Z}) \cong \mb{Z}$ is pure, $\mr{sp}_2$ is surjective.
It follows from the definition of the Gysin morphism that
\[\Ima(f_2: H^2(\widetilde{X}_0 \cap \mb{P}^1, \mb{Z}) \to H^2(\mb{P}^1, \mb{Z}) \oplus H^2(\widetilde{X}_0, \mb{Z}))=\mb{Z}([\mb{P}^1]^\vee \oplus f'),\]
hence does not intersect $\mb{Z}(0 \oplus f')$ non-trivially (here $f_2$
as in Corollary \ref{ntor03}).
Using Corollary \ref{ntor03}, this implies $\Ima(\mr{sp}_2)=\mr{sp}_2(\mb{Z}(0 \oplus f'))$.
As $\phi_s$ is an isomorphism, this implies $\ov{\phi}_s(0 \oplus f')$ 
is the Poincar\'{e} dual of the fundamental class of $\mc{X}_s$ (up to a sign), denoted $f_s$. Since pullback commutes with cup-product, Theorem \ref{ntor11}
implies that $e_{1,s},...,e_{2g,s}$ is a symplectic basis of $H^1(\mc{X}_s,\mb{Z})$ with $e_{i,s}e_{i+g,s}=-f_s$ for $1 \le i \le g$.
Using \cite[Remark $5.3$]{kingn}, $H^i(\mc{G}(2,\mc{L})_s,\mb{Q})$ has a $\mb{Q}$-basis consisting of
 monomials of the form 
$\alpha_s^{j_1}\beta_s^{j_2}\psi_{i_1,s}\psi_{i_2,s}...\psi_{i_k,s}$ such that 
  $j_1+k<g$, $j_2+k<g$, $i_1<i_2<...<i_k$ and $2j_1+4j_2+3k=i$.
  By Theorems \ref{ntor11} and \ref{ner01}, $\psi_{2g}^\infty$ (resp. $\psi_g^\infty$) generates $\mr{Gr}^W_4H^3(\mc{G}(2,\mc{L})_\infty, \mb{Q})$
  (resp. $W_2H^3(\mc{G}(2,\mc{L})_\infty, \mb{Q})$). Since cup-product is a morphism of mixed Hodge structures,
  we have a basis of $W_iH^i(\mc{G}(2,\mc{L})_\infty, \mb{Q})$ consisting
   of monomials of the form 
  $\alpha_\infty^{j_1}\beta_\infty^{j_2}\psi^\infty_{i_1}\psi^\infty_{i_2}...\psi^\infty_{i_k}$
  with $j_1+k<g$, $j_2+k<g$, $i_1<i_2<...<i_k$, $2j_1+4j_2+3k=i$ satisfying: if $2g \in \{i_1,...,i_k\}$ then $g \in \{i_1,...,i_k\}$.
  Using the isomorphism $\oplus \eta_j$, we then obtain the following commutative diagram 
\[\begin{diagram}
 0&\rTo&\bigoplus_{k=0}^g P_k^\infty \otimes I_{g-k}^\infty&\rTo&\bigoplus_{k=0}^g P_k^\infty \otimes \mb{Q}[\alpha_\infty, \beta_\infty, \psi_\infty]&\rTo^{\nu_\infty}&H^*(\mc{G}(2,\mc{L})_\infty,\mb{Q})\\
   & &\dTo^{\oplus \eta_j}&\circlearrowleft&\dTo^{\oplus \eta_j}&\circlearrowleft&\dTo^{\oplus \eta_j}&\\
   0&\rTo&\bigoplus_{k=0}^g P_k^{(s)} \otimes I_{g-k,s}&\rTo&\bigoplus_{k=0}^g P_k^{(s)} \otimes \mb{Q}[\alpha_s, \beta_s, \psi_s]&\rTo^{\nu}&H^*(\mc{G}(2,\mc{L})_s,\mb{Q})&\rTo&0
  \end{diagram}\]
where $\nu_\infty$ is the natural morphism and the two rows are exact.
Using the description of $W_iH^i(\mc{G}(2,\mc{L})_\infty, \mb{Q})$ above, it is easy to check that
$\Ima(\nu_\infty)=W_iH^i(\mc{G}(2,\mc{L})_\infty, \mb{Q})$. Therefore, 
\begin{equation}\label{eq:bet10}
\bigoplus_iW_iH^i(\mc{G}(2,\mc{L})_\infty, \mb{Q}) \cong  \bigoplus\limits_i P_i^\infty \otimes \frac{\mb{Q}[\alpha_\infty, \beta_\infty, \psi_\infty]}{I^\infty_{g-i}}.
\end{equation}

Recall by Corollary \ref{ntor03} that $\ker(\mr{sp}_i) \cong \mr{Im}(f_i) \cong H^{i-2}(\mc{G}_0 \cap \mc{G}_1,\mb{Q})/ \ker(f_i)$.
Then using the identification \eqref{ntor04} along with Theorem \ref{bet01} we conclude:
{\small \[\ker(\mr{sp}_i) \cong \frac{H^{i-2}(M_{\widetilde{X}_0}(2,\widetilde{\mc{L}}_0),\mb{Q}) \oplus H^{i-4}(M_{\widetilde{X}_0}(2,\widetilde{\mc{L}}_0),\mb{Q})(\mb{Q}\xi_1 \oplus \mb{Q}\xi_2) 
 \oplus H^{i-6}(M_{\widetilde{X}_0}(2,\widetilde{\mc{L}}_0),\mb{Q})\xi_1\xi_2}{H^{i-4}(M_{\widetilde{X}_0}(2,\widetilde{\mc{L}}_0),\mb{Q})(\xi_1 \oplus -\xi_2)}.
\]}
Hence, $\oplus_i \ker(\mr{sp}_i) \cong H^*(M_{\widetilde{X}_0}(2,\widetilde{\mc{L}}_0),\mb{Q})(-2)[X,Y]/(X^2,Y^2,X-Y)$ defined by 
sending $\xi_1$ (resp. $\xi_2$) to $X$ (resp. $Y$), where $(X^2, Y^2, X-Y)$ is the ideal in the ring 
$H^*(M_{\widetilde{X}_0}(2,\widetilde{\mc{L}}_0),\mb{Q})[X,Y]$ generated $X^2, Y^2$ and $X-Y$. Using \cite[Theorem $3.2$]{kingn}, we conclude that 
\[\bigoplus\limits_i \ker(\mr{sp}_i) \cong \bigoplus\limits_i P_{i-2} \otimes \frac{\mb{Q}[\alpha, \beta, \psi, X, Y]}{(I_{g-i-3}, X^2, Y^2, X-Y)}.\]
By Corollary \ref{ntor03}, we have $H^*(\mc{G}_{X_0}(2,\mc{L}_0),\mb{Q})
\cong (\oplus_i W_iH^i(\mc{G}(2,\mc{L})_\infty, \mb{Q})) \oplus (\oplus_i \ker(\mr{sp}_i)).$
The theorem follows immediately.
 \end{proof}

\section{Hodge-Poincar\'{e} formula}\label{sec5}
The Hodge-Poincar\'{e} formula for moduli spaces of semi-stable sheaves
on smooth, projective curves is well-known and was classically computed
by Earl and Kirwan \cite{EK}. It records the Hodge decomposition of the cohomology ring of the 
moduli space. In this section, we compute the Hodge-Poincar\'{e} formula
for the Gieseker's moduli space of rank $2$ semi-stable sheaves with fixed determinant. One important difference with 
the classical case is the that the cohomology of the Gieseker's moduli
space is not pure, which make computations more complicated.

We use notations as in Notation \ref{tor33} and \S \ref{subsec1}. We briefly discuss the idea of the proof of Theorem \ref{th:giespo}.
Fix $s \in \Delta^*$.
The first step is to use the isomorphism $\eta_{_i}$ as in \eqref{ntor15}, to prove that there exists $\xi \in H^{2,1}(\mc{G}(2,\mc{L})_s, \mb{C})$ such that 
{\small \[H^3(\mc{G}(2,\mc{L})_s,\mb{C}) \cong \eta_{_3}(\mr{Gr}^W_{3}H^{3}(\mc{G}(2,\mc{L})_\infty,\mb{C})) \oplus \mb{C}\xi \oplus \mb{C}\ov{\xi} \mbox{ and }
 \{\xi^2, \ov{\xi}^2, \xi\ov{\xi} \} \in \eta_{_6}(\mr{Gr}^W_{6}H^{6}(\mc{G}(2,\mc{L})_\infty,\mb{C})).\]}
Next, we use Newstead's classical result \cite[Theorem $1$]{new1}, which states that the cohomology ring of $\mc{G}(2,\mc{L})_s$
is generated is degrees $2, 3$ and $4$. Since cup-product is a morphism of mixed Hodge structures, we get Proposition \ref{prop:bet05}.
We then use the isomorphism $\Phi_1$ to prove that $\mr{Gr}^W_iH^{i+1}(\mc{G}(2, \mc{L})_\infty, \mb{Q})$ can be
identified with $H^{i-2}(M_{\widetilde{X}_0}(2,\widetilde{\mc{L}}_0), \mb{Q})$ as pure Hodge structures. Theorem \ref{th:giespo} then follows
from the exact sequence \eqref{ntor02} combined with Theorem \ref{bet01}.

\begin{prop}\label{prop:bet05}
For any $s \in \Delta^*$, we have the following equality:
\[h^{p,q}\mr{Gr}^W_{p+q}H^{p+q}(\mc{G}(2,\mc{L})_\infty,\mb{C})=h^{p,q}(\mc{G}(2,\mc{L})_s)-h^{p-2,q-1}(M_{\widetilde{X}_0}(2,\widetilde{\mc{L}}_0))-h^{p-1,q-2}(M_{\widetilde{X}_0}(2,\widetilde{\mc{L}}_0)).\]
\end{prop}

\begin{proof}
For any $s \in \Delta^*$, we have the natural isomorphism 
 \[\eta_{_i}:H^i(\mc{G}(2,\mc{L})_\infty, \mb{Q}) \xrightarrow{\sim} H^i(\mc{G}(2,\mc{L})_s, \mb{Q})\]
for all $i \ge 0$, as  in \eqref{ntor15}.
Given a subspace $W$ of $H^i(\mc{G}(2,\mc{L})_s, \mb{Z})$ (resp. $H^i(\mc{G}(2,\mc{L})_\infty, \mb{Q})$), denote by 
$W^{T_{\mc{G}(2,\mc{L})_s}}$ (resp. $W^{T_{\mc{G}(2,\mc{L})}}$) the subspace of $W$ consisting of elements that are invariant under 
the action of the monodromy operator $T_{\mc{G}(2,\mc{L})_s}$ (resp. $T_{\mc{G}(2,\mc{L})}$), where the monodromy operators are described in 
\eqref{int01} and the discussion before that.
We claim that there exists $\xi \in F^2H^3(\mc{G}(2,\mc{L})_s,\mb{C})$ such that $T_{\mc{G}(2,\mc{L})_s}(\xi) \not= \xi$.
Indeed, $H^3(\mc{G}(2,\mc{L})_s,\mb{C})=F^2 H^3(\mc{G}(2,\mc{L})_s,\mb{C}) \oplus \ov{F^2 H^3(\mc{G}(2,\mc{L})_s,\mb{C})}$.
Since $T_{\mc{G}(2,\mc{L})_s}$ commutes with conjugation, we observe that if the entire 
$F^2H^3(\mc{G}(2,\mc{L})_s,\mb{C})$ is invariant under $T_{\mc{G}(2,\mc{L})_s}$, then so is $H^3(\mc{G}(2,\mc{L})_s,\mb{C})$.
Since $T_{\mc{G}(2,\mc{L})}$ is a canonical extension of $T_{\mc{G}(2,\mc{L})_s}$, this implies $H^3(\mc{G}(2,\mc{L})_\infty,\mb{C})$ is 
$T_{\mc{G}(2,\mc{L})}$-invariant. But, Propositions \ref{tor26} and \ref{ner15} imply that $\mr{Gr}^W_4H^3(\mc{G}(2,\mc{L})_\infty,\mb{C}) \cong \mb{C}$,
which is not $T_{\mc{G}(2,\mc{L})}$-invariant by the invariant cycle theorem (Remark \ref{ntor01}). This proves the claim.

Since $\mr{Gr}^W_iH^i(\mc{G}(2,\mc{L})_\infty, \mb{C})$ is $T_{\mc{G}(2,\mc{L})}$-invariant, 
\cite[p. $66$, Lemma $2.4.12$ and p. $69$, Theorem $6.6$]{kuli}
implies that $\eta_{_i}(H^{p,i-p} \mr{Gr}^W_i H^i(\mc{G}(2,\mc{L})_\infty, \mb{C})) 
\subset H^{p,i-p} (\mc{G}(2,\mc{L})_s)$.
Since 
\[\mb{C} \cong \mr{Gr}^W_4 H^3(\mc{G}(2,\mc{L})_\infty,\mb{C}) \cong  F^2 \mr{Gr}^W_4 H^3(\mc{G}(2,\mc{L})_\infty,\mb{C})
\mbox{ and }\] $H^{p,p} \mr{Gr}^W_{2p}(\mc{G}(2,\mc{L})_\infty,\mb{C})=
H^{2p}(\mc{G}(2,\mc{L})_\infty,\mb{C})$ for $p=1, 2$ (by Remark \ref{tor25}, any one dimensional limit mixed Hodge structure is pure),
this implies $\eta_{_{2p}}(H^{2p}(\mc{G}(2,\mc{L})_\infty,\mb{C}))=
H^{p,p}(\mc{G}(2,\mc{L})_s,\mb{C})$ for $p=1,2$ and 
$F^2H^3(\mc{G}(2,\mc{L})_s,\mb{C})=\eta_{_3}(F^2\mr{Gr}^W_3H^3(\mc{G}(2,\mc{L})_\infty,\mb{C})) \oplus \mb{C}\xi$, 
with $\xi$ as before.
Therefore, $\ov{F^2H^3(\mc{G}(2,\mc{L})_s,\mb{C})}=\eta_{_3}(\ov{F^2\mr{Gr}^W_3H^3(\mc{G}(2,\mc{L})_\infty,\mb{C})}) \oplus \mb{C}\ov{\xi}$,
where $\ov{\xi}$ denotes the conjugate of $\xi$.
Let $\psi_i^\infty$ for $1 \le i \le 2g$ be the generators of $H^3(\mc{G}(2,\mc{L})_\infty,\mb{C})$ as defined in \S \ref{sec4}. Recall, by 
Theorems \ref{ntor11} and \ref{ner01}, we have $\mr{Gr}^W_4H^3(\mc{G}(2,\mc{L})_\infty,\mb{Q})=\mb{Q}\psi_{2g}^\infty$ and
$\mr{Gr}^W_2H^3(\mc{G}(2,\mc{L})_\infty,\mb{Q})=\mb{Q}\psi_g^\infty$. As 
$\mr{Gr}^W_3H^3(\mc{G}(2,\mc{L})_\infty,\mb{C}) \cong F^2\mr{Gr}^W_3H^3(\mc{G}(2,\mc{L})_\infty,\mb{C}) \oplus \ov{F^2\mr{Gr}^W_3H^3(\mc{G}(2,\mc{L})_\infty,\mb{C})}$,
we can assume that $\xi=\eta_{_3}(\lambda_1 \psi_g^\infty+\lambda_2\psi_{2g}^\infty)$ for some $\lambda_1, \lambda_2 \in \mb{C}$.
Since $T_{\mc{G}(2,\mc{L})_s}(\xi) \not= \xi$, $\lambda_2 \not= 0$ ($\psi_g^\infty$ is $T_{\mc{G}(2,\mc{L})}$-invariant).
Replace $\xi$ by $\xi/\lambda_2$. Then, $\xi=\eta_{_3}(\lambda \psi_g^\infty+\psi_{2g}^\infty)$ for some $\lambda \in \mb{C}$.
Since $(\psi_g^\infty)^2=0=(\psi_{2g}^\infty)^2$
and $\psi_g^\infty\psi_{2g}^\infty \in \mr{Gr}^W_6H^6(\mc{G}(2,\mc{L})_\infty,\mb{C})$,  
$\xi^2, \ov{\xi}^2$ and $\xi\ov{\xi}$ must belong to $\eta_{_6}(\mr{Gr}^W_6H^6(\mc{G}(2,\mc{L})_\infty,\mb{C}))$.
Using the isomorphism $\Phi_1$ (see \S \ref{sec4}) of pure Hodge structures from $\mr{Gr}^W_3H^3(\mc{G}(2,\mc{L})_\infty,\mb{C})$
to $H^3(M_{\widetilde{X}_0}(2,\widetilde{\mc{L}}_0),\mb{C})$, we can then conclude that 
\[H^{p,q}(\mc{G}(2,\mc{L})_s,\mb{C}) \cong \eta_{_{p+q}}(H^{p,q} \mr{Gr}^W_{p+q}H^{p+q}(\mc{G}(2,\mc{L})_\infty,\mb{C})) \oplus 
 H^{p-2,q-1}(M_{\widetilde{X}_0}(2,\widetilde{\mc{L}}_0),\mb{C})\xi \oplus\]\[\oplus  
 H^{p-1,q-2}(M_{\widetilde{X}_0}(2,\widetilde{\mc{L}}_0),\mb{C})\ov{\xi}\]
 (use cup-product is a morphism of Hodge structures, $\xi$ is of type $(2,1)$ 
 and $H^*(\mc{G}(2,\mc{L})_s, \mb{Q})$ is generated in degrees $2, 3$ and $4$ as mentioned before).
This proves the proposition.
 \end{proof}

  \begin{thm}\label{th:giespo}
   The Hodge-Poincar\'{e} formula for the cohomology ring $H^*(\mc{G}_{X_0}(2,\mc{L}_0), \mb{C})$ is
   \[\frac{(1+xy^2)^{g-1}(1+x^2y)^{g-1}(1+2xy)(1+x^2y^2)-x^gy^g(1+x)^{g-1}(1+y)^{g-1}(2+(1+xy)^2)}{(1-xy)(1-x^2y^2)}.\]
  \end{thm}

  \begin{proof}
   Using Theorem \ref{bet01} and the identification \eqref{ntor04}, the exact sequence \eqref{ntor02} become the following short exact sequence:
  \begin{equation}\label{eq:bet05}
   0 \to \bigoplus\limits_{j=1}^3 H^{p-j,q-j}(M_{\widetilde{X}_0}(2,\widetilde{\mc{L}}_0),\mb{C})(-j) \xrightarrow{f_i} 
   H^{p,q} \mr{Gr}^W_{p+q} H^{p+q}(\mc{G}_{X_0}(2,\mc{L}_0),\mb{C}) 
  \end{equation}
  \[\xrightarrow{\mr{sp}_{p+q}} H^{p,q} \mr{Gr}^W_{p+q} H^{p+q}(\mc{G}(2,\mc{L})_\infty, \mb{C}) \to 0\] and $H^{p,q} \mr{Gr}^W_{p+q} H^{p+q+1}(\mc{G}_{X_0}(2,\mc{L}_0),\mb{C}) \cong H^{p,q} \mr{Gr}^W_{p+q} H^{p+q+1}(\mc{G}(2,\mc{L})_\infty, \mb{C})$.
Using \cite[Remark $5.3$]{kingn}, $\mr{Gr}^W_{p+q}H^{p+q+1}(\mc{G}(2,\mc{L})_\infty,\mb{Q})$ has a $\mb{Q}$-basis consisting of
 monomials of the form 
\[\alpha_\infty^{j_1}\beta_\infty^{j_2}\psi_g^\infty\psi_{i_1}^\infty \psi_{i_2}^\infty...\psi_{i_k}^\infty \mbox{ such that } 
\{g, 2g\} \cap \{i_1,...,i_k\}=\emptyset,
  j_1+k<g-1, j_2+k<g-1,\] $i_1<i_2<...<i_k$ and $2j_1+4j_2+3(k+1)=p+q+1$ (cup-product is a morphism of mixed Hodge structures). 
  Define the morphism 
 \[\Phi_{p+q+1}:\mr{Gr}^W_{p+q}H^{p+q+1}(\mc{G}(2,\mc{L})_\infty,\mb{Q}) 
  \to H^{p+q-2}(M_{\widetilde{X}_0}(2,\widetilde{\mc{L}}_0), \mb{Q})
 \mbox{ as }\]
\[\Phi_{p+q+1}(\alpha_\infty^{j_1}\beta_\infty^{j_2}\psi_g^\infty\psi_{i_1}^\infty\psi_{i_2}^\infty...\psi_{i_k}^\infty)=\alpha^{j_1}\beta^{j_2}\Phi_1(\psi_{i_1}^\infty)\Phi_1(\psi_{i_2}^\infty)...\Phi_1(\psi_{i_k}^\infty)\]
  and extend linearly, where $\Phi_1$ is the isomorphism defined in \S \ref{sec4}. Since $\psi_g^\infty$ is of Hodge type $(1,1)$ and $\Phi_1$ is an isomorphism of Hodge structures, 
\cite[Remark $5.3$]{kingn} implies that $\Phi_{p+q+1}$ is an isomorphism of pure Hodge structures which sends Hodge type $(i,p+q-i)$ to $(i-1, p+q-1-i)$. Note that the $(p,q)$-th part of the 
cohomology ring $H^*(\mc{G}_{X_0}(2,\mc{L}_0), \mb{C})$
is given by
\[h^{p,q}(\mc{G}_{X_0}(2,\mc{L}_0), \mb{C}):=h^{p,q}\mr{Gr}^W_{p+q}H^{p+q}(\mc{G}_{X_0}(2,\mc{L}_0), \mb{C}) + h^{p,q}\mr{Gr}^W_{p+q}H^{p+q+1}(\mc{G}_{X_0}(2,\mc{L}_0), \mb{C}).\]
Using the isomorphism $\Phi_{p+q+1}$, the exact sequence \eqref{eq:bet05} and Proposition \ref{prop:bet05} we have
\[h^{p,q}(\mc{G}_{X_0}(2,\mc{L}_0))=\left(h^{p,q}(\mc{G}(2,\mc{L})_s)-h^{p-2,q-1}(M_{\widetilde{X}_0}(2,\widetilde{\mc{L}}_0))-h^{p-1,q-2}(M_{\widetilde{X}_0}(2,\widetilde{\mc{L}}_0))+\right.\]
\[\left. +\sum\limits_{j=1}^3 h^{p-j,q-j}(M_{\widetilde{X}_0}(2,\widetilde{\mc{L}}_0))\right)+h^{p-1,q-1}(M_{\widetilde{X}_0}(2,\widetilde{\mc{L}}_0)).\]
By \cite[Corollary $2.9$]{bano}, the Hodge-Poincar\'{e} formulas for $\mc{G}(2,\mc{L})_s$ and $M_{\widetilde{X}_0}(2,\widetilde{\mc{L}}_0)$
are given by 
\begin{equation}\label{eq:bet06}
P_g(x,y):=\frac{(1+xy^2)^g(1+x^2y)^g-x^gy^g(1+x)^g(1+y)^g}{(1-xy)(1-x^2y^2)} \mbox{ and } Q(x,y):=P_{g-1}(x,y),
\end{equation}
respectively. This implies, the Hodge-Poincar\'{e} formula for 
$\mc{G}_{X_0}(2,\mc{L}_0)$ is given by 
\[P_g(x,y)+Q(x,y)(-x^2y-xy^2+2xy+x^2y^2+x^3y^3)=\]
\[=\frac{(1+xy^2)^{g-1}(1+x^2y)^{g-1}(1+2xy)(1+x^2y^2)-x^gy^g(1+x)^{g-1}(1+y)^{g-1}(2+(1+xy)^2)}{(1-xy)(1-x^2y^2)}.\]
This proves the theorem.
\end{proof}

\section{Simpson's moduli space with fixed determinant}\label{sec6}
 In this section 
we prove the analogue of the Mumford conjecture for the Simpson's moduli space of rank $2$ semi-stable sheaves  with fixed odd degree determinant on an irreducible nodal curve
and compute the associated Hodge-Poincar\'{e} formula (Theorem \ref{th:simppo}). We use Notation \ref{tor33} and notations as in \S \ref{note:ner01}.

\subsection{Simpson's moduli spaces with fixed determinant}
  Let $\mc{E}$ be a rank $2$, torsion-free sheaf on $X_0$ of degree $d$ and $\mc{L}_0$ an invertible sheaf on $X_0$ of degree $d$.
  We say that $\mc{E}$ \emph{has determinant} $\mc{L}_0$ if 
  there is a $\mathcal{O}_{_{X_0}}$-morphism $\wedge^2(\mc{E})\to \mc{L}_0$ which is an isomorphism outside the node $x_0$.
Note that if $\mc{E}$ is locally free then this condition implies
$\wedge^2\mc{E}\cong \mc{L}_0$.

Let $U_{X_0}(2,d)$ be the moduli space of stable rank $2$, degree $d$ torsion free sheaves on $X_0$ (see \cite{sesh}).
Denote by $U_{X_0}(2,d)^{0}$ the sublocus 
parameterizing locally free sheaves. Note that $U_{X_0}(2,d)^{0}$ is an open subvariety of $U_{X_0}(2,d)$. 
We have a well defined morphism $\mbox{det}: U_{X_0}(2,d)^{0}\to \mr{Pic}(X_0)$ defined
by $\mc{E} \mapsto \wedge^2\mc{E}$. Denote by $U_{X_0}(2,\mc{L}_0)^{0}:=\mbox{det}^{-1}([\mc{L}_0])$.

 Denote by $U_{X_0}(2,\mc{L}_0):=\{[\mc{E}]\in U_{X_0}(2,d) \mid \mc{E} \mbox{ has determinant } \mc{L}_0\}$. 
By \cite[Theorem 1.10]{sun1}, the Zariski closure $\overline{U_{X_0}(2,\mc{L}_0)^{0}}$ of $U_{X_0}(2,\mc{L}_0)^{0}$ in $U_{X_0}(2,d)$ is $U_{X_0}(2,\mc{L}_0)$.

\subsection{Stratification on the  moduli space}
Let $m_{x_0}$ denote the maximal ideal of $\mo_{X,x_0}$. In \cite{bhos2}, Bhosle shows 
there exists a stratification $U_0\subset U_1 \subset U_2:=U_{X_0}(2,d)$ of 
$U_{X_0}(2,d)$ by locally closed subschemes, where
\[ U_i:=\{[\mc{E}]\in U_{X_0}(2,d) \mid \mc{E}_{x_0} \simeq \mathcal{O}_{x_0}^{\oplus j} \oplus m_{x_0}^{\oplus 2-j} \mbox{ for } j \le i\}.\]
This induces the stratification 
$U_0(\mc{L}_0)\subset U_1(\mc{L}_0)\subset U_{X_0}(2,\mc{L}_0)$, where $U_i(\mc{L}_0):=U_i \cap U_{X_0}(2,\mc{L}_0)$.

Denote by $\pi:\widetilde{X}_0 \to X_0$ the normalization morphism.
By \cite[Proposition 4.9]{Bhosle}, there exists a natural isomorphism from $M_{_{\widetilde{X}_0}}(2,d-2)$ to $U_0$, sending $[\mc{E}]$ to $[\pi_*\mc{E}]$.
Denote by $\widetilde{\mc{L}}_0:=\pi^*\mc{L}_0$ and $D:=\pi^{-1}(x_0)$. 
Then, $M_{\widetilde{X}_0}(2,\widetilde{\mc{L}}_0(-D))$ maps isomorphically to $U_0(\mc{L}_0)$ (see proof of \cite[$(6.1)$]{tha}).

\subsection{Comparison between Gieseker's and Simpson's moduli spaces}
Recall, there exists a natural proper morphism 
\begin{equation}\label{eq:bet11}
 \theta:\mc{G}_{X_0}(2,\mc{L}_0)\to U_{X_0}(2,\mc{L}_0)
\end{equation}
defined by 
pushing forward a rank $2$, locally-free sheaf defined over a curve semi-stably equivalent to $X_0$, via the natural contraction map to $X_0$
(see \cite[Theorem $3.7(3)$]{sun1}). Denote by $\mc{G}_{X_0}(2,\mc{L}_0)^0 \subset \mc{G}_0$ the sub-locus of $\mc{G}_{X_0}(2,\mc{L}_0)$ 
parameterizing locally-free sheaves on $X_0$.
Using \cite[Remark $5.2$]{sesh4}, we conclude that $\theta$ maps the irreducible component $\mc{G}_1$
into $U_0(\mc{L}_0)$ and $\mc{G}_{X_0}(2,\mc{L}_0)^0$ maps isomorphically to $U_{X_0}(2,\mc{L}_0)^0$ (use the description of the irreducible components of 
$\mc{G}_{X_0}(2,\mc{L}_0)$ as given in \cite[\S $6$]{tha} and \cite[Theorem $3.7$]{sun1}). 
By the properness of $\theta$, we note that $\theta$ maps $\mc{G}_1$ surjectively to $U_0(\mc{L}_0)$.
Moreover, since $U_0(\mc{L}_0) \cong M_{\widetilde{X}_0}(2,\widetilde{\mc{L}}_0(-D))$, it is non-singular (see \cite[Corollary $4.5.5$]{huy}).

 \subsection{Generalized Mumford's conjecture and Hodge-Poincar\'{e} formula for Simpson's moduli space}
We briefly discuss the idea of the proof of Theorem \ref{th:simppo}.
 Using the restriction of the proper morphism $\theta$ to $\mc{G}_1$  and the identifications \eqref{ntor04} and \eqref{ntor05},
we compute the kernel and the cokernel of the pull-back morphism $\theta^*$ (see Proposition \ref{prop:bet02}).
Combining \eqref{ntor02} and Proposition \ref{ner15}, we obtain an explicit 
description of the kernel of the specialization morphism $\mr{sp}_i$.
We use this to show that $H^i(U_{X_0}(2,\mc{L}_0),\mb{Q})$ can be identified with 
the image of $\mr{sp}_i$ as mixed Hodge structures. Theorem \ref{th:simppo} then follows from 
Theorems \ref{th:bet1} and \ref{th:giespo}.

\begin{prop}\label{prop:bet02}
The natural morphism $\theta^*:H^i(U_{X_0}(2,\mc{L}_0),\mb{Q}) \to H^i(\mc{G}_{X_0}(2,\mc{L}_0),\mb{Q})$ is injective 
with $\mr{Gr}^W_{i-1}H^i(U_{X_0}(2,\mc{L}_0),\mb{Q}) \xrightarrow[\sim]{\theta^*} \mr{Gr}^W_{i-1} H^i(\mc{G}_{X_0}(2,\mc{L}_0),\mb{Q})$
and cokernel isomorphic to $\bigoplus\limits_{j=1}^3 H^{i-2j}(M_{\widetilde{X}_0}(2,\widetilde{\mc{L}}_0), \mb{Q})(\xi'')^j$
such that the resulting (cokernel) morphism factors as 
 \begin{equation}\label{eq:bet03}
 \mr{Gr}^W_iH^i(\mc{G}_{X_0}(2,\mc{L}_0),\mb{Q}) \xrightarrow{r^*} H^i(\mc{G}_1, \mb{Q}) \xrightarrow{\eqref{ntor05}} \bigoplus\limits_{j=1}^3 H^{i-2j}(M_{\widetilde{X}_0}(2,\widetilde{\mc{L}}_0), \mb{Q})(\xi'')^j,
 \end{equation}
where $r: \mc{G}_1 \to \mc{G}_{X_0}(2,\mc{L}_0)$ is the natural inclusion.
 \end{prop}

\begin{proof}
  Recall, $U_{X_0}(2,\mc{L}_0)\backslash U_0(\mc{L}_0) \cong \mc{G}_{X_0}(2,\mc{L}_0)\backslash \mc{G}_1 \cong \mc{G}_0 \backslash \mc{G}_0 \cap \mc{G}_1$.
  Since $\mc{G}_0$   and $\mc{G}_1$ are an 
  excessive couple (\cite[Example B.$5(2)$]{pet}), we have for all $i \ge 0$,
  \[H^i(\mc{G}_{X_0}(2,\mc{L}_0),\mc{G}_1) \cong H^i(\mc{G}_0, \mc{G}_1 \cap \mc{G}_0) \cong H^i_c(\mc{G}_0 \backslash \mc{G}_0 \cap \mc{G}_1) \cong  H^i(U_{X_0}(2,\mc{L}_0), U_0(\mc{L}_0)),\]
  where the last two isomorphisms follow from \cite[Corollary B.$14$]{pet}.
  The morphism $\theta$ induces the following commutative
  diagram where every morphism is a morphism of mixed Hodge structures (use \cite[Proposition $5.46$]{pet}):
 {\small \begin{equation}\label{eq:bet04}
  \begin{diagram}
  H^i(U_{X_0}(2,\mc{L}_0), U_0(\mc{L}_0))&\rTo&H^i(U_{X_0}(2,\mc{L}_0),\mb{Q})&\rTo&H^i(U_0(\mc{L}_0),\mb{Q})&\rTo&H^{i+1}(U_{X_0}(2,\mc{L}_0), U_0(\mc{L}_0))\\
  \dTo^{\cong}&\circlearrowleft&\dTo^{\theta^*}&\circlearrowleft&\dTo^{(\theta')^*}&\circlearrowleft&\dTo^{\cong}\\
  H^i(\mc{G}_{X_0}(2,\mc{L}_0),\mc{G}_1)&\rTo&H^i(\mc{G}_{X_0}(2,\mc{L}_0),\mb{Q})&\rTo^{r^*}&H^i(\mc{G}_1,\mb{Q})&\rTo&H^{i+1}(\mc{G}_{X_0}(2,\mc{L}_0),\mc{G}_1)
    \end{diagram}
    \end{equation}}
Using the identification \eqref{ntor05} and $U_0(\mc{L}_0) \cong M_{\widetilde{X}_0}(2,\widetilde{\mc{L}}_0)$, 
we have the following short exact sequence of pure Hodge structures:
\[0 \to H^i(U_0(\mc{L}_0),\mb{Q}) \xrightarrow{(\theta')^*} H^i(\mc{G}_1,\mb{Q}) \to \bigoplus\limits_{j=1}^3 H^{i-2j}(M_{\widetilde{X}_0}(2,\widetilde{\mc{L}}_0), \mb{Q})(\xi'')^j \to 0.\]
  Using the identifications \eqref{ntor04} and \eqref{ntor05} observe that the image of the composition 
  \[H^{i-2}(\mc{G}_0 \cap \mc{G}_1,\mb{Q})(-1) \xrightarrow{f_i} \mr{Gr}^W_i H^i(\mc{G}_{X_0}(2, \mc{L}_0),\mb{Q}) \xrightarrow{r^*} H^i(\mc{G}_1,\mc{Q})\]
  is isomorphic to $\bigoplus\limits_{j=1}^3 H^{i-2j}(M_{\widetilde{X}_0}(2,\widetilde{\mc{L}}_0), \mb{Q})(\xi'')^j$ with 
  kernel $H^{i-4}(M_{\widetilde{X}_0}(2,\widetilde{\mc{L}}_0),\mb{Q})(-2)$ (see proof of Theorem \ref{bet01}). This implies that the natural morphism 
  \[H^{i-2}(\mc{G}_0 \cap \mc{G}_1, \mb{Q}) \xrightarrow{r^* \circ f_i} H^i(\mc{G}_1, \mb{Q}) \to \mr{coker}((\theta')^*)\]
  is surjective. Since this morphism factors through $\mr{coker}(\theta^*)$ (use the diagram \eqref{eq:bet04}), we conclude that the natural morphism from $\mr{coker}(\theta^*)$ to 
  $\mr{coker}((\theta')^*)$ is a surjective. By diagram chase of \eqref{eq:bet04}, we have the natural morphism from $\mr{coker}(\theta^*)$
 to $\mr{coker}((\theta')^*)$ is injective, hence an isomorphism. 
  Using the diagram \eqref{eq:bet04} and the analogous diagram after replacing $i$ by $i-1$, one can observe by diagram chase that $\theta^*$ is injective.
 The proposition then follows immediately.
 \end{proof}
 
 \begin{thm}\label{th:simppo}
 Let $P_i^\infty$ and $I_k^\infty$ be as in \S \ref{sec4}. Then, 
 \[H^*(U_{X_0}(2,\mc{L}_0),\mb{Q}) \cong  \bigoplus\limits_i P_i^\infty \otimes \frac{\mb{Q}[\alpha_\infty, \beta_\infty, \psi_\infty]}{I^\infty_{g-i}}.\]
 Moreover, the Hodge-Poincar\'{e} polynomial associated to the moduli space $U_{X_0}(2,\mc{L}_0)$ is
  \[\frac{(1+xy^2)^{g-1}(1+x^2y)^{g-1}(1+xy+x^3y^3)-x^gy^g(1+x)^{g-1}(1+y)^{g-1}(2+xy)}{(1-xy)(1-x^2y^2)}.\]
 \end{thm}

 \begin{proof}
 Observe that by Corollary \ref{ntor03} the image of the composition 
 {\small \begin{equation}\label{eq:fin01}
  \Ima(f_i) = \ker(\mr{sp}_i) \hookrightarrow H^i(\mc{G}_{X_0}(2,\mc{L}_0),\mb{Q}) \xrightarrow{r^*} H^i(\mc{G}_1,\mb{Q}) \xrightarrow[\sim]{\eqref{ntor05}}\bigoplus\limits_{j=0}^3 H^{i-2j}(M_{\widetilde{X}_0}(2,\widetilde{\mc{L}}_0), \mb{Q})(\xi'')^j
 \end{equation}}
coincides with the image of the Gysin morphism from 
$H^{i-2}(\mc{G}_0 \cap \mc{G}_1,\mb{Q})$ to $H^i(\mc{G}_1,\mb{Q})$.
Under the notations of \S \ref{note:ner01}, the Gysin morphism 
from $H^*(\mb{P}^1 \times \mb{P}^1, \mb{Q})$ to 
$H^{*+2}(\p3, \mb{Q})$ sends $1$ (resp. $\xi_1+\xi_2, \xi_1\xi_2$)
to $\xi''$ (resp. $(\xi'')^2, (\xi'')^3$), up to multiplication by 
a non-zero rational number. Using the identification 
\eqref{ntor04}, this implies that the image of \eqref{eq:fin01} is isomorphic to \[\bigoplus\limits_{j = 1}^3 H^{i-2j}(M_{\widetilde{X}_0}(2,\widetilde{\mc{L}}_0), \mb{Q})(\xi'')^j.\]
   Proposition \ref{prop:bet02} then implies that the morphism 
 {\small \[\ker(\mr{sp}_i: H^i(\mc{G}_{X_0}(2,\mc{L}_0), \mb{Q}) \to H^i(\mc{G}(2,\mc{L})_\infty, \mb{Q})) \to \mr{coker}(\theta^*: H^i(U_{X_0}(2,\mc{L}_0), \mb{Q}) \to H^i(\mc{G}_{X_0}(2,\mc{L}_0), \mb{Q}))\]}
 induced by $\ker(\mr{sp}_i) \hookrightarrow H^i(\mc{G}_{X_0}(2,\mc{L}_0), \mb{Q}) \to H^i(\mc{G}_{X_0}(2,\mc{L}_0), \mb{Q})/\Ima(\theta^*)=
 \mr{coker}(\theta^*)$,  is an isomorphism of pure Hodge structures.
  Therefore, by Corollary \ref{ntor03} we have 
  \begin{equation}\label{eq:bet12}
   H^i(U_{X_0}(2,\mc{L}_0),\mb{Q}) \cong W_i H^i(\mc{G}(2,\mc{L})_\infty, \mb{Q}) \mbox{ as mixed Hodge structures}
 \end{equation}
 ($H^i(U_{X_0}(2,\mc{L}_0),\mb{Q})$ 
 has weights at most $i$ as $U_{X_0}(2,\mc{L}_0)$ is compact). The first part of the theorem
 then follows directly from \eqref{eq:bet10}. 
 
 Note that the $(p,q)$-th part of the 
cohomology ring $H^*(U_{X_0}(2,\mc{L}_0), \mb{C})$ is given by
\[h^{p,q}(U_{X_0}(2,\mc{L}_0), \mb{C}):=h^{p,q}\mr{Gr}^W_{p+q}H^{p+q}(U_{X_0}(2,\mc{L}_0), \mb{C}) + h^{p,q}\mr{Gr}^W_{p+q}H^{p+q+1}(U_{X_0}(2,\mc{L}_0), \mb{C}).\]
In the proof of Theorem \ref{th:giespo} we constructed an isomorphism of pure Hodge structures 
\[\Phi_{p+q+1}: \mr{Gr}^W_{p+q} H^{p+q+1}(\mc{G}(2,\mc{L})_\infty, \mb{C}) \to H^{p+q-2}(M_{\widetilde{X}_0}(2,\widetilde{\mc{L}}_0), \mb{C})\]
sending 
a class of type $(p,q)$ to that of type $(p-1,q-1)$. 
Then, Proposition 
\ref{prop:bet05} along with \eqref{eq:bet12} implies that 
 \[h^{p,q}(U_{X_0}(2,\mc{L}_0), \mb{C})=h^{p,q}(\mc{G}(2,\mc{L})_s, \mb{C})-
 h^{p-2,q-1}(M_{\widetilde{X}_0}(2,\widetilde{\mc{L}}_0), \mb{C})-h^{p-1,q-2}(M_{\widetilde{X}_0}(2,\widetilde{\mc{L}}_0), \mb{C})+\]\[+h^{p-1,q-1}(M_{\widetilde{X}_0}(2,\widetilde{\mc{L}}_0), \mb{C}).
\]
Let $P_g(x,y)$ and $Q(x,y)$ be the Hodge-Poincar\'{e} polynomial $\mc{G}(2,\mc{L})_s$ and $M_{\widetilde{X}_0}(2,\widetilde{\mc{L}}_0)$ respectively, defined in \eqref{eq:bet06}. Then, the Hodge-Poincar\'{e} polynomial of $U_{X_0}(2,\mc{L}_0)$ is given by 
\[P_g(x,y)+Q(x,y)(-x^2y-xy^2+xy)=\]
\[=\frac{(1+xy^2)^{g-1}(1+x^2y)^{g-1}(1+xy+x^3y^3)-x^gy^g(1+x)^{g-1}(1+y)^{g-1}(2+xy)}{(1-xy)(1-x^2y^2)}.\]
This proves the theorem.
 \end{proof}

 \begin{rem}\label{rem:gen}
  Under the identification \eqref{eq:bet12} above, the cohomology ring $H^*(U_{X_0}(2,\mc{L}_0), \mb{Q})$
  is generated by $\alpha_\infty, \beta_\infty$, $\psi_i^\infty$ for $1 \le i \le 2g-1$ and $\psi_g^\infty\psi_{2g}^\infty$.
 \end{rem}

\end{document}